\documentclass[11pt]{amsart}

\usepackage{amsmath,amsthm, amscd, amssymb, amsfonts}

\newcommand{\gyd}{{^{\ku C_{m^{2}}}_{\ku C_{m^{2}}}\mathcal{YD}}}

\def\ra{\rightarrow}
\def\va{\varepsilon}

\def\a{\alpha}
\def\b{\beta}

\def\l{\lambda}

\newcommand{\Tw}{{\text{Tw}}}

\newcommand{\otb}{{\overline{\otimes}}}
\newcommand{\otk}{{\otimes_{\ku}}}
\newcommand{\Mo}{{\mathcal M}}
\newcommand{\No}{{\mathcal N}}
\newcommand{\mo}{{\mathcal M}}
\newcommand{\nic}{{\mathfrak B}}

\newcommand{\fs}{{\mathfrak f}}
\newcommand{\js}{{\mathfrak J}}

\newcommand{\ot}{{\otimes}}

\newcommand{\kc}{{\mathcal K}}
\newcommand{\Ac}{{\mathcal A}}

\newcommand{\ca}{{\mathcal C}}

\newcommand{\Do}{{\mathcal D}}

\newcommand{\Fc}{{\mathcal F}}
\newcommand{\Gc}{{\mathcal G}}

\newcommand{\op}{\rm{op}}

\newcommand{\ku}{{\Bbbk}}

\newcommand{\Na}{{\mathbb N}}

\newcommand{\uno}{{\mathbf 1}}

\newcommand{\sy}{{\mathbb S}}

\newcommand{\id}{\mbox{\rm id\,}}

\newcommand{\Aut}{\mbox{\rm Aut\,}}

\newcommand{\Rad}{\mbox{\rm Rad\,}}

\newcommand{\Vect}{\mbox{\rm Vect\,}}

\newcommand{\Mod}{\mbox{\rm Mod\,}}

\newcommand{\Fun}{\operatorname{Hom}}

\newcommand\Rep{\operatorname{Rep}}

\newcommand\Hom{\operatorname{Hom}}

\newcommand{\End}{\operatorname{End}}

\newcommand{\Id}{\mathop{\rm Id}}

\newcommand{\gr}{\mbox{\rm gr\,}}

\renewcommand{\_}[1]{\mbox{$_{\left( #1 \right)}$}}

\theoremstyle{plain}

\numberwithin{equation}{section}

\newtheorem{teo}{Theorem}[section]

\newtheorem{lema}[teo]{Lemma}

\newtheorem{cor}[teo]{Corollary}

\newtheorem{prop}[teo]{Proposition}

\theoremstyle{definition}

\newtheorem{defi}[teo]{Definition}

  \newtheorem{exa}[teo]{Example}

\theoremstyle{remark}

\newtheorem{rmk}[teo]{Remark}

\def\pf{\begin{proof}}

\def\epf{\end{proof}}

\theoremstyle{remark}

\begin{document}

\title[Module categories over  finite pointed tensor categories]{Module categories over  finite pointed  tensor categories}
\author[Galindo, Mombelli]{C\'esar Galindo and
Mart\'\i n Mombelli }
\thanks{ The work of M.M. was  supported by
 CONICET, Secyt (UNC), Mincyt (C\'ordoba) Argentina}
\address{ Departamento de Matem\'aticas, Universidad de los Andes,
\newline \indent Carrera 1 N${}^\circ$ 18A - 12  Bogot\'a, Colombia}
\email{cn.galindo1116@uniandes.edu.co, cesarneyit@gmail.com
\newline \indent\emph{URL:}\/ http://matematicas.uniandes.edu.co/$\sim$cgalindo/}
\address{ Facultad de Matem\'atica, Astronom\'\i a y F\'\i sica, \newline
\indent Universidad Nacional de C\'ordoba, \newline \indent Medina Allende s/n,
(5000) Ciudad Universitaria, \newline \indent C\'ordoba, Argentina}
\email{martin10090@gmail.com, mombelli@mate.uncor.edu
\newline \indent\emph{URL:}\/ http://www.mate.uncor.edu/$\sim$mombelli}

\begin{abstract} We study exact module categories over the representation categories of
 finite-dimensional quasi-Hopf algebras. As a consequence we classify exact module categories
 over some families of pointed tensor categories with cyclic group of invertible objets of order $p$,
 where $p$ is a prime number.

\bigbreak
\bigbreak


\end{abstract}

\date{\today}
\maketitle

\section*{Introduction}

For a given tensor category $\ca$ a \emph{module category} over $\ca$, or a $\ca$-\emph{module}, is the categorification of the notion of module over a ring, it consist of an Abelian category $\Mo$ together with a biexact functor $\ot:\ca\times \Mo\to \Mo$ satisfying natural associativity and unit axioms. A module category $\Mo$ is \emph{exact} \cite{eo} if for any projective object $P\in \ca$ and any $M\in \Mo$ the object $P\otimes M$ is again projective.

\medbreak

The notion of module category has been used with profit in the theory of tensor categories,
 see \cite{DGNO},\cite{ENO1}, \cite{ENO2}. Interestingly, the notion of module categories is related
 with diverse areas of mathematics and mathematical physics such as subfactor theory \cite{Oc}, \cite{BEK};
 extensions of vertex algebras \cite{KO}, Calabi-Yau algebras \cite{Gi},  Hopf algebras \cite{N},
affine Hecke algebras \cite{BO} and conformal field theory, see for example
 \cite{BFS}, \cite{CS1}, \cite{CS2}, \cite{FS1}, \cite{FS2}, \cite{O1}.

\medbreak

 The classification of exact module categories over a given tensor category was undertaken by several authors:
 \begin{itemize}
  \item[1.] When $\ca$ is the semisimple quotient of $U_q(\mathfrak{sl}_2)$ \cite{Oc},
\cite{KO}, \cite{eo2},
\item[2.] over the category of finite-dimensional $SL_q(2)$-comodules \cite{O3},
        \item[3.] over the tensor categories of representations of finite
supergroups \cite{eo},
  \item[4.] for any group-theoretical tensor category
\cite{O2},
 \item[5.] over the Tambara-Yamagami categories \cite{Ga2}, \cite{MM},
\item[6.]  over the Hageerup fusion categories \cite{GS},

  \item[7.] over $\Rep(H)$, where $H$ is a lifting of a
quantum linear space \cite{Mo2}.
\end{itemize}

In this paper we are concerned with the classification of exact module categories over some families of finite non-semisimple pointed tensor categories that are not equivalent to the representation categories of Hopf algebras.

\medbreak

An object $X$ in a tensor category is invertible if there is another object $Y$ such that $X\ot Y\simeq \uno \simeq Y\ot X$. A \emph{pointed tensor category} is a tensor category such that every simple object is invertible. The invertible objects form a group. Pointed tensor categories with cyclic group of invertible objects were studied in \cite{EG1}, \cite{EG2}, \cite{EG3} and later in \cite{A}.

\medbreak

Any finite pointed tensor category is equivalent to the representation category of a finite-dimensional
 quasi-Hopf algebra $A$. In the case when the group of invertible elements is a cyclic group $G$ there
 exists an action of $G$ on $\Rep(A)$ such that the equivariantization $\Rep(A)^G$ is equivalent to the
 representation category of a finite-dimensional pointed Hopf algebra $H$, see \cite{A}.
The purpose of this work is to relate module categories over $\Rep(A)$ and module categories over
 $\Rep(H)$ and whenever is possible obtain a classification of exact module categories over $\Rep(A)$
 assuming that we know the classification for $\Rep(H)$. Module categories over any quasi-Hopf algebra
 are parameterized by Morita equivariant equivalence classes of comodule algebras. We would like to establish
 a correspondence as follows:
\begin{equation*}
\left\{
\begin{array}{c}
\mbox{Morita quivalence classes}\\
\mbox{of $H$-comodule algebras } \\
\mbox{such that $G\subseteq K_0$} \\
\end{array}
\right\}
\begin{array}{c}
{}\\
\xrightarrow{\quad  \quad }\\
\xleftarrow{\quad  \quad}\\
{}
\end{array}
\left\{
\begin{array}{c}
\mbox{Morita equivalence classes}\\
\mbox{of $A$-comodule algebras} \\
\mbox{$(\kc, \Phi_\lambda)$}\\
\end{array}
\right\}.
\end{equation*}
\medbreak

The contents of the paper are the following. In Section \ref{s:modcat} we recall the notion of exact module category, the notion of tensor product of module categories over a tensor category. In Section \ref{s:equi} we recall the notion of $G$-graded tensor categories, $G$-actions of tensor categories and crossed products tensor categories. We also recall the  $G$-equivariantization construction of tensor categories
 and module categories.

Section \ref{s:quasi} is devoted to study comodule algebras over quasi-Hopf algebras and
 how they give rise to module categories. Next, in Section \ref{s:equivarin-quasi} we study the
 equivariantization of the representation category of a quasi-Hopf algebra and the
  equivariantization of comodule algebras. We describe the datum that gives rise to an action in a
 representation category of a comodule algebra, that we call a \emph{crossed system} and we
 prove that the equivariantization of module categories are modules over a certain crossed product
 comodule algebra.

In Section \ref{s:basicquasi} we recall the definition of a family of finite-dimensional basic
 quasi-Hopf algebras introduced by I. Angiono \cite{A} that are denoted by $A(H,s)$, where
 $H$ is a coradically graded Hopf algebra with cyclic group of group-like elements. A particular class of these  quasi-Hopf algebras
were introduced by S. Gelaki \cite{Ge}  and later used by Etingof and Gelaki to classify certain families of pointed tensor
categories. There is an action
 of a group $G\subseteq G(H)$ on $\Rep(A(H,s))$ such that $\Rep(A(H,s))^G\simeq \Rep(H)$ \cite{A}.
 For any left $H$-comodule algebra $K$ such that $\ku G\subseteq K_0$ we construct a left
 $A(H,s)$-comodule algebra. We prove that in the case that $\mid G(H)\mid=p^2$,
 where $p$ is a prime number, the representation category of this family of comodule algebras is big enough to contain all module
 categories over $\Rep(A(H,s))$. We apply this result to classify module categories in the case when
 $H$ is the bosonization of a quantum linear space.

\subsection*{Acknowledgments} We are very grateful to Iv\'an Angiono  for many
 fruitful conversations and for patiently answering our questions on his work \cite{A}. We also thank  the
referee for his constructive comments.

\section{Preliminaries and notation}

Hereafter $\ku$ will denote an algebraically closed field of characteristic 0. All vector spaces and algebras will be considered over $\ku$.

If $H$ is a Hopf algebra and $A$ is an $H$-comodule algebra via $\lambda:A\to H\otk A$, we
shall say that a (right) ideal $J$ is $H$-costable if
$\lambda(J)\subseteq H\otk J$. We shall say that $A$ is (right)
$H$-simple, if there is no nontrivial (right) ideal $H$-costable
in $A$.

\medbreak

If $H$ is a finite-dimensional Hopf algebra then $H_0\subseteq H_1
\subseteq \dots \subseteq H_m=H$ will denote the coradical
filtration. When $H_0\subseteq H$ is a Hopf subalgebra then the
associated graded algebra $\gr H$ is a coradically graded Hopf
algebra. If $(A, \lambda)$ is a left $H$-comodule algebra, the
coradical filtration on $H$ induces a filtration on $A$, given by
$A_n=\lambda^{-1}(H_n\otk A)$ called the \emph{Loewy filtration}.

\subsection{Finite tensor categories and tensor functors}
A \emph{ tensor category over} $\ku$ is a $\ku$-linear Abelian rigid monoidal category. A
\emph{finite tensor category} \cite{eo} is a tensor category such that  it has a finite number of isomorphism classes
 of simple objects, Hom spaces are finite-dimensional
$\ku$-vector spaces, all objects have finite length, every simple object has a projective cover and
the unit object is simple.

Hereafter all tensor categories will be considered over $\ku$ and every functor will be assumed to be
$\ku$-linear.

If $\ca$, $\Do$ are tensor categories, the collection $(F,\xi,\phi):\ca\to \Do$ is
a \emph{tensor functor} if $F :\ca\to \Do$ is a functor, $\phi:F(\uno_\ca)\to \uno_\Do$ is an isomorphism and for any $X, Y\in\ca$ the family of natural isomorphisms $\zeta_{X,Y}:F(X)\ot F(Y)\to F(X\ot Y)$ satisfies
\begin{equation}\label{tensorf1}\zeta_{X,Y\ot Z} (\id_{F(X)}\ot
\zeta_{Y,Z})a_{F(X),F(Y),F(Z)}= F(a_{X,Y,Z})\zeta_{X\ot
Y,Z}(\zeta_{X,Y}\ot\id_{F(Z)}),
\end{equation}
\begin{equation}
\label{tensorf2}l_{F(X)}=F(l_X)\zeta_{\uno,X} (\phi\ot\id_{F(X)}),
\end{equation}
\begin{equation} \label{tensorf3}r_{F(X)}
=F(r_X)\zeta_{X,\uno}(\id_{F(X)}\ot\phi),
\end{equation}

If $(F, \zeta), (G, \xi):\ca\to \Do$ are tensor functors, a \textit{natural tensor transformation} $\gamma:F\to G$ is a natural transformation such that $\gamma_{X\ot Y} \zeta_{X,Y}= \xi_{X,Y} (\gamma_X\ot \gamma_Y)$ for all $X, Y\in \ca$.

\section{Module categories}\label{s:modcat}

 A (left) \emph{module category} over a tensor
category $\ca$ is an Abelian category $\Mo$ equipped with an exact
bifunctor $\otb: \ca \times \Mo \to \Mo$, that we will sometimes refer as the \emph{action},  natural associativity
and unit isomorphisms $m_{X,Y,M}: (X\otimes Y)\otb M \to X\otimes
(Y\otb M)$, $\ell_M: \uno \otb M\to M$ subject to natural
associativity and unity axioms. See for example \cite{eo}. A module category $\Mo$ is \emph{exact}, \cite{eo},  if for any projective object
$P\in \ca$ the object $P\otb M$ is projective in $\Mo$ for all
$M\in\Mo$.
Sometimes we shall also say  that $\Mo$ is a $\ca$-module. Right module categories and bimodule categories are defined similarly.

\medbreak

If $\Mo$ is a left $\ca$-module then $\Mo^{\op}$ is the  right $\ca$-module over the opposite Abelian
 category with action $\Mo^{\op}\times \ca\to \Mo^{\op}$, $(M, X)\mapsto X^*\otb M$ and associativity isomorphisms $m^{\op}_{M,X,Y}=m_{Y^*, X^*, M}$ for all $X, Y\in \ca, M\in \Mo$.
\medbreak

If  $\ca, \ca', \mathcal{E}$ are tensor categories, $\Mo$ is a $(\ca,\mathcal{E})$-bimodule category and $\No$ is an $(\mathcal{E},\ca')$-bimodule
category, we shall denote the tensor product over $\mathcal{E}$  by $\Mo\boxtimes_\mathcal{E}\No$. This category is a $(\ca,\ca')$-bimodule category. For more details on the tensor product of module categories the reader is referred to \cite{ENO3}, \cite{Gr}.

A module functor between module categories $\Mo$ and $\Mo'$ over a
tensor category $\ca$ is a pair $(T,c)$, where $T:\Mo \to
\Mo'$ is a  functor and $c_{X,M}: T(X\otb M)\to
X\otb T(M)$ is a natural isomorphism such that for any $X, Y\in
\ca$, $M\in \Mo$:
\begin{align}\label{modfunctor1}
(\id_X\otimes c_{Y,M})c_{X,Y\otb M}T(m_{X,Y,M}) &=
m_{X,Y,T(M)}\, c_{X\otimes Y,M}
\\\label{modfunctor2}
\ell_{T(M)} \,c_{\uno ,M} &=T(\ell_{M}).
\end{align}
We shall use the notation $(T, c): \Mo \to \Mo'$. There is a composition
of module functors: if $\Mo''$ is another module category and
$(U,d): \Mo' \to \Mo''$ is another module functor then the
composition
\begin{equation}\label{modfunctor-comp}
(U\circ T, e): \Mo \to \Mo'', \qquad \text {where } e_{X,M} = d_{X,U(M)}\circ
U(c_{X,M}),
\end{equation} is
also a module functor.

\smallbreak  Let $\Mo_1$ and $\Mo_2$ be module categories over
$\ca$. We denote by $\Hom_{\ca}(\Mo_1, \Mo_2)$ the category whose
objects are module functors $(\Fc, c)$ from $\Mo_1$ to $\Mo_2$. A
morphism between  $(\Fc,c)$ and $(\Gc,d)\in\Hom_{\ca}(\Mo_1,
\Mo_2)$ is a natural transformation $\alpha: \Fc \to \Gc$ such
that for any $X\in \ca$, $M\in \Mo_1$:
\begin{gather}
\label{modfunctor3} d_{X,M}\alpha_{X\otb M} =
(\id_{X}\otb \alpha_{M})c_{X,M}.
\end{gather}

  Two module categories $\Mo_1$ and $\Mo_2$ over $\ca$
are {\em equivalent} if there exist module functors $F:\Mo_1\to
\Mo_2$ and $G:\Mo_2\to \Mo_1$ and natural isomorphisms
$\id_{\Mo_1} \to F\circ G$, $\id_{\Mo_2} \to G\circ F$ that
satisfy \eqref{modfunctor3}.

\medbreak

The  direct sum of two module categories $\Mo_1$ and $\Mo_2$ over
a tensor category $\ca$  is the $\ku$-linear category $\Mo_1\times
\Mo_2$ with coordinate-wise module structure. A module category is
{\em indecomposable} if it is not equivalent to a direct sum of
two non trivial module categories.

\medbreak

If $(F, \xi):\ca\to\ca$ is a tensor functor and $(\Mo, \otb, m)$ is a module
category over $\ca$ we shall denote by $\Mo^F$ the module category
$(\mo, \otb^F, m^F)$ with the same underlying Abelian category with action and associativity
 isomorphisms defined by
$$X\otb^F M=F(X)\otb M,\quad m_{X,Y,M}^F=m_{F(X),F(Y),M} (\xi^{-1}_{X,Y}\otb\, \id_M), $$
for all $X, Y\in \ca$, $M\in \Mo$.

\section{Equivariantization of tensor categories}\label{s:equi}

\subsection{Group actions on tensor categories}

We briefly recall the group actions on tensor categories and the equivariantization construction. For more details the reader is referred to \cite{DGNO}.

\medbreak

Let $\ca$ be a tensor category and let $\underline{\text{Aut}_\otimes(\ca)}$
 be the monoidal category of tensor auto-equivalences of $\ca$, arrows are  tensor natural
 isomorphisms and tensor product the  composition of monoidal functors. We shall denote  by
 $\text{Aut}_\otimes (\ca)$ the group of isomorphisms classes of tensor auto-equivalences of  $\ca$, with the multiplication induced
 by the composition, \emph{i.e.}  $[F][F']= [F\circ F']$.

\medbreak

For any group $G$ we shall denote by $\underline{G}$ the monoidal category where objects are elements of $G$ and tensor product is given by the product of $G$. An action of the group  $G$ over a   $\ca$, is a
monoidal functor   $*:\underline{G}\to \underline{\text{Aut}_\otimes(\ca)}$.
In another words for any $\sigma\in G$ there is a tensor functor $(F_\sigma, \zeta_\sigma):\ca\to\ca$, and for any $\sigma,\tau\in G$, there are natural tensor isomorphisms $\gamma_{\sigma,\tau}:F_\sigma\circ F_\tau\to F_{\sigma\tau}$.

\subsection{$G$-graded tensor categories}
 Let $G$ be a group and $\ca$ be a tensor category. We
shall say that  $\ca$ is  $G$-graded, if there is a decomposition
$$\ca=\oplus_{\sigma\in G} \ca_\sigma$$ of $\ca$ into a direct sum of full
Abelian subcategories, such that  for all $\sigma, \tau\in G$, the
bifunctor  $\otimes$ maps $\ca_\sigma\times \ca_\tau$ to  $\ca_{\sigma
\tau}$. Given a $G$-graded tensor category $\ca$, and a subgroup $H\subset G$, we shall denote by $\ca_H$ the tensor subcategory $\oplus_{h\in H}\ca_h$.

\subsection{$G$-equivariantization of  tensor categories} Let $G$ be a group acting on a  tensor category $\ca$. An \emph{equivariant} object in $\ca$ is a pair $(X, u)$ where $X\in \ca$ is an object together
 with isomorphisms $u_\sigma:F_\sigma(X)\to X$ satisfying
$$
u_{\sigma \tau}\circ (\gamma_{\sigma,\tau})_X=u_\sigma\circ F_\sigma(u_\tau),
$$
for all $\sigma,\tau \in G$. A
$G$-equivariant morphism $\phi: (V, u) \to (W, u')$ between $G$-equivariant objects $(V, f)$ and $(W, \sigma)$, is a morphism $\phi: V \to W$ in
$\mathcal M$ such that $\phi\circ u_\sigma = u'_\sigma\circ F_\sigma(\phi)$ for all $\sigma \in G$.

The tensor category of equivariant objects is denoted by $\ca^G$ and it is
called the \emph{equivariantization} of $\ca$. The tensor product of $\ca^G$ is defined by \begin{align*}
    (V, u)\otimes (W, u'):= (V\otimes W, \tilde{u}),
\end{align*}where $\tilde{u}_\sigma= (u_\sigma\ot u'_\sigma) \zeta_\sigma^{-1},$
 for any $\sigma\in G$. The unit object is $(1, \text{id}_1)$.

\subsection{Crossed product tensor categories and $G$-invariant module categories}\label{section:Ggraded}

Given an action $*:\underline{G}\to
\underline{\text{Aut}_\otimes(\ca)}$ of $G$ on $\ca$,  the
$G$-crossed product tensor category, denoted by $\ca\rtimes G$ is
defined as follows. As an Abelian category $\ca\rtimes G=
\bigoplus_{\sigma\in G}\ca_\sigma$, where $\ca_\sigma =\ca$ as an
Abelian category, the tensor product is
$$[X, \sigma]\otimes [Y,\tau]:= [X\otimes F_\sigma(Y),
\sigma\tau],\  \  \   X,Y\in \ca,\  \   \sigma,\tau\in G,$$ and
the unit object is $[1,e]$. See \cite{Ta} for the associativity
constraint and a proof of the pentagon identity.

If $\ca=\Rep(A)$ is the  representation category of a finite-dimensional quasi-Hopf algebra $A$ then $\ca\rtimes G$ is also a representation category of a finite-dimensional quasi-Hopf algebra $B$. This is an immediate consequence of \cite[Prop. 2.6]{eo} since each simple object $W\in \ca\rtimes G$ is isomorphic to $[V,e]\otimes[1,\sigma]$, where $\sigma \in G$ and $V\in \Rep(A)$ is simple. Let $d: K_0(\ca)\to \mathbb Z$ the Perron-Frobenius dimension, then $d([V,e]\otimes[1,\sigma])=d(V)d([1,\sigma])=d(V)\in \mathbb Z$, where $d([1,\sigma])=1$ because $[1,\sigma]$ is multiplicatively invertible.

\subsection{Equivariantization of module categories}
We shall explain analogous procedures for
equivariantization in module categories. Equivariant module categories appeared in \cite{ENO2}.
 We shall use the approach given in \cite{Ga1}.

\medbreak

Let $G$ be a group and $\ca$ be a tensor category
equipped with an
action of $G$. Let $\Mo$ be a module category over $\ca$. For any $g\in G$ we shall denote by
 $\Mo^\sigma$ the module category $\Mo^{F_\sigma}$. If $\sigma\in G$, we shall say that an endofunctor
 $T:\Mo\to\Mo$ is  $\sigma$-\emph{invariant} if it has a module structure $(T,c):\Mo\to\Mo^\sigma$.
\medbreak

If $\sigma, \tau\in G$ and $T$ is $\sigma$-invariant and $U$ is $\tau$-invariant then $T\circ U$ is $\sigma\tau$-invariant. Indeed, let us assume that the functors  $(T,c):\Mo\to\Mo^\sigma$, $(U,d):\Mo\to\Mo^\tau$ are module functors then $(T\circ U,b):\Mo\to\Mo^{\sigma\tau}$ is a module functor, where
\begin{align}\label{invariant-composition} b_{X,M}=((\gamma_{\sigma,\tau})_X\ot\id)c_{F_\tau(X),M}T(d_{X,M}),
\end{align}
for all $X\in \ca$, $M\in \Mo$.

\begin{defi}Let $F\subseteq G$ be a subgroup. \begin{itemize}

\item[1.] The monoidal category of $\sigma$-equivariant functors for some $\sigma\in F$ in $\Mo$ will be denoted by $\underline{\Aut^F_{\ca}(\Mo)}$.

\item[3.] An $F$-\emph{equivariant} module category is a module category $\Mo$ equipped with a monoidal functor $(\Phi,\mu):\underline{F}\to \underline{\Aut^F_{\ca}(\Mo)}$, such that $\Phi(\sigma)$ is  a $\sigma$-invariant functor for any $\sigma\in F$.
             \end{itemize}
\end{defi}

In another words, an $F$-equivariant module category is a module category $\Mo$ endowed with a family of module functors $(U_\sigma, c^\sigma):\Mo\to \Mo^\sigma$ for any $\sigma\in F$ and a family of natural isomorphisms $\mu_{\sigma,\tau}:(U_\sigma\circ U_\tau, b)\to (U_{\sigma\tau}, c^{\sigma \tau})$ $\sigma, \tau\in F$ such that
\begin{equation}\label{mod-equi1} {(\mu_{\sigma,\tau \nu})}_M\circ U_\sigma( \mu_{\tau,\nu})_M={(\mu_{\sigma \tau,\nu})}_M\circ {(\mu_{\sigma,\tau})}_{U_\nu(M)},
\end{equation}
\begin{equation}\label{mod-equi2} c^{\sigma\tau}_{X,M}\circ(\mu_{\sigma,\tau})_{X\otb M}= ((\gamma_{\sigma,\tau})_X\otb (\mu_{\sigma,\tau})_M)\circ c^\sigma_{F_\tau(X),U_\tau(M)}\circ
U_\sigma(c^\tau_{X,M}),
\end{equation}
for all $\sigma,\tau,\nu\in F$, $X\in \ca$, $M\in \Mo$. Equation \eqref{mod-equi1} follows from \eqref{tensorf1} and \eqref{mod-equi2} follows from \eqref{modfunctor3}.

\begin{exa} $\ca$ is a $G$-equivariant module category over itself. For any $g\in G$ set $(U_\sigma, c^\sigma)=(F_\sigma, \theta_\sigma)$ and $\mu_{\sigma,\tau}=\gamma_{\sigma,\tau}$ for all $\sigma, \tau\in G$.
\end{exa}

If $\Mo$ is an $F$-equivariant module category, an \emph{equivariant object} (see \cite[Def. 5.3]{ENO2}) is an object $M\in \Mo$ together with isomorphisms $\{v_\sigma:U_\sigma(M)\to M: \sigma\in F\}$ such that for all $\sigma, \tau\in F$
\begin{align}\label{F-equivariant-obj} v_{\sigma\tau} \circ (\mu_{\sigma,\tau})_M=v_\sigma\circ U_\sigma(v_\tau).
\end{align}
The category of $F$-equivariant objects is denoted by $\Mo^F$. A morphism between two $F$-equivariant objects $(M,v)$, $(M',v')$ is a morphism $f:M\to M'$ in $\Mo$ such that $ f\circ v_\sigma= v'_\sigma\circ U_\sigma(f)$ for all $\sigma\in F$.

\begin{lema}\label{equivariant-mod-cat} The category $\Mo^F$ is a $\ca^G$-module category.
\end{lema}
\pf If $(X, u)\in \ca^G$ and $(M,v)\in\Mo^F$ the action is defined by
$$ (X, u)\otb (M,v)=(X\otb M, \widetilde{v}),$$
where $\widetilde{v}_\sigma=(u_\sigma\ot v_\sigma)c^\sigma_{X,M}$ for all $\sigma\in F$. The object $(X\otb M, \widetilde{v})$ is equivariant due to equation \eqref{mod-equi2}. The associativity isomorphisms are the same as in $\Mo$.
\epf

The notion of $F$-equivariant module category is equivalent to the notion of $\ca\rtimes F$-module cateory. If $\Mo$ is an $F$-equivariant $\ca$-module category for some subgroup $F$ of $G$, then $\Mo$ is a $\ca\rtimes F$-module with action $\otb: \ca\rtimes F\times \Mo\to \Mo$ given by
$[X,g]\otb M=X\otb U_g(M)$, for all $X\in \ca$, $g\in F$ and $M\in \Mo$. The associativity isomorphisms are given by
$$ m_{[X,g], [Y,h], M}=\big(\id_{X}\ot (c^g_{Y, U_h(M)})^{-1}(\id_{F_g(Y)}\ot \mu_{g,h}^{-1}(M))\big)\, m_{X, F_g(Y), U_{gh}(M)},$$
for all $X, Y\in \ca$, $g,h\in F$ and $M\in \Mo$.

In the next statement we collect several well-known results that are, by now, part of the folklore of the subject.

\begin{prop}\label{mod-cat-equi} Let $G$ be a finite group acting over a finite tensor category $\ca$. If $F\subset G$ is a subgroup, and $\Mo$ is an $F$-equivariant $\ca$-module category, then:

\begin{itemize}

\item[1.] If $\Mo$ is an exact ( indecomposable) $\ca$-module category then $\Mo$ is an exact (respectively indecomposable) $\ca\rtimes F$-module category.

               \item[2.] $\Mo^F$ is an exact  module category
               if and only if $\Mo$ is an exact  module category.

   \item[3.]      There is an equivalence of $\ca^G$-module categories \begin{equation}\label{equivariant-equivalences} \Mo^F\simeq \Fun_{\ca\rtimes F}(\ca, \Mo) \simeq \ca^{\op} \boxtimes_{\ca\rtimes F} \Mo\simeq \big(\ca \rtimes G \boxtimes_{\ca\rtimes F}  \Mo\big)^G.
       \end{equation}

\item[4.] If $\No$ is an indecomposable (exact) module category
over $\ca^G$ there exists a subgroup $F$ of $G$ and an $F$-equivariant indecomposable (exact)
module category $\Mo$ over $\ca$   such that $\No\simeq \Mo^F.$

\item[5.] If $\Mo_1$, $\Mo_2$ are $G$-equivariant $\ca$-module categories such that $\Mo^G_1\simeq \Mo^G_2$ as $\ca^G$-module categories then $\Mo_1\simeq \Mo_2$ as $\ca$-module categories.
             \end{itemize}
\end{prop}
\pf

1. Let $P\in \ca\rtimes G$ be a projective object. Thus, there exists a family of projective objects  $P_\sigma \in \ca$ such that $P = \oplus_{\sigma \in G} [P_\sigma,\sigma]$. Let $M\in \Mo$, then $P\overline{\otimes}M = \bigoplus_{\sigma \in G} P_\sigma \overline{\otimes} U_\sigma (M)$, and since $\Mo$ is an exact $\ca$-module category $P_\sigma \overline{\otimes} U_\sigma (M)$ is projective for all $\sigma$, thus $P\overline{\otimes}M$ is projective.

\smallbreak

2. Under the correspondence described in \cite[Thm. 4.1]{Ta} is enough to show that a $\ca\rtimes F$-module category $\Mo$ is exact if and only if $\Mo$ is an exact $\ca$-module category. The proof follows from part (1) of this proposition.
\smallbreak

3. An object $(F,c)\in  \Fun_{\ca\rtimes F}(\ca, \Mo)$ is determined uniquely  by an object $M\in \Mo$ such that $F(X)=X\otb M$ together with an isomorphism $v_\sigma=c_{[\uno, \sigma],\uno}:U_\sigma(M)\to M$. This correspondence establish an equivalence $\Mo^F\simeq \Fun_{\ca\rtimes F}(\ca, \Mo)$. The equivalence $\Fun_{\ca\rtimes F}(\ca, \Mo)\simeq \ca^{\op} \boxtimes_{\ca\rtimes F} \Mo$ follows from \cite[Thm. 3.20]{Gr}.

Since $\ca \rtimes G \boxtimes_{\ca\rtimes F}  \Mo$ is a $\ca \rtimes G$-module then it is a $G$-equivariant $\ca$-module category, thus
$$\big(\ca \rtimes G \boxtimes_{\ca\rtimes F}  \Mo\big)^G\simeq \ca^{\op} \boxtimes_{\ca\rtimes G} \big( \ca \rtimes G \boxtimes_{\ca\rtimes F}  \Mo\big) \simeq \ca^{\op} \boxtimes_{\ca\rtimes F}  \Mo\simeq \Mo^F. $$

The first equivalence is \cite[Thm 4.1]{Ta}.

 \smallbreak

4. By  \cite[Proposition 3.9]{eo} every indecomposable exact tensor category over a finite tensor category is a simple module category in the sense of \cite{Ga1}, so  the result follows by the main result of  \cite{Ga1}, and the item (1) of this proposition.
\smallbreak

5. Since $\Mo_1$, $\Mo_2$ are $G$-equivariant then they are $\ca\rtimes G$-module categories. It follows from \cite[Thm. 4.1]{Ta} that this are equivalent $\ca\rtimes G$-module categories. This equivalence induces an equivalence of $\ca$-module categories (see \cite[Ex. 2.5]{Ta}).\epf

It follows from Proposition \ref{mod-cat-equi} (4) that the  equivariantization construction of module categories
 by a fixed subgroup is injective. Moreover, if the equivariantization of a module category by two subgroups gives
 the same result then the groups must be conjugate.
 We shall give the precise statement in the following. First we need  a definition and a result from the paper \cite{Ga2}.

\begin{defi}\cite[Def. 4.3]{Ga2}
Let $\ca$ be a $G$-graded tensor category. If $(\Mo,\otimes)$ is a
$\ca_e$-module category, then a $\ca$-extension of $\Mo$ is a
$\ca$-module category $(\Mo,\odot)$ such that $(\Mo,\otimes)$ is
obtained by restriction to $\ca_e$.
\end{defi}

\begin{prop}\label{prop equivalent inducidas}\cite[Prop. 4.6]{Ga2}
Let $\ca$ be a $G$-graded finite tensor category and let $F,F'\subset G$ be subgroups and $(\No,\odot)$, $(\No',\odot')$
be a $\ca_F$-extension and a $\ca_{F'}$-extension of the
indecomposable $\ca_e$-module categories $\No$ and $\No'$,
respectively. Then $\ca\boxtimes_{\ca_{F'}}\No' \cong
\ca\boxtimes_{\ca_F}\No$ as $\ca$-modules if and only if there exists $\sigma \in G$
such that $F=\sigma F'\sigma^{-1}$ and $\ca_{\sigma
F'}\boxtimes_{\ca_{ F'}}\No'\cong \No$ as $\ca_e$-module categories.
\end{prop}

\begin{teo} Let $G$ be a finite group action on a finite tensor category  $\ca$ and let $F,F'\subset G$ be subgroups.
 Let $\No$ and $\No'$ be an $F$-equivariant and an $F'$-equivariant
module categories respectively, such that $\No$ and $\No'$ are
indecomposable as $\ca$-module categories and $\No^F\cong \No'^{F'}$ as
$\ca^G$-module categories. Then $F$ and $F'$ are conjugate subgroups
in $G$.
\end{teo}
\begin{proof} It follows from Proposition \ref{mod-cat-equi} (3) that there is an equivalence of $\ca^G$-modules
$$ \big(\ca \rtimes G \boxtimes_{\ca\rtimes F}  \No\big)^G \simeq \big(\ca \rtimes G \boxtimes_{\ca\rtimes F'}  \No'\big)^G.$$
Hence by Proposition \ref{mod-cat-equi} (5) there is an equivalence of  $\ca \rtimes G$-modules $\ca \rtimes G \boxtimes_{\ca\rtimes F}  \No\simeq \ca \rtimes G \boxtimes_{\ca\rtimes F'}  \No'$, thus the result follows from Proposition \ref{prop equivalent inducidas}.
\end{proof}

\section{Quasi-Hopf algebras}\label{s:quasi}

A quasi-bialgebra \cite{D} is a four-tuple $(A, \Delta , \va , \Phi )$ where $A$ is
an associative algebra with unit,
$\Phi\in (A\ot A\ot A)^{\times}$ is called the \emph{associator}, and
$\Delta :\ A\ra A\ot A$, $\va : A\ra k$ are algebra
homomorphisms satisfying the identities
\begin{eqnarray}
&&\Phi (\Delta \ot \id)(\Delta (h))=(\id \ot \Delta )(\Delta (h))\Phi
,\label{q1}\\
&&(\id \ot \va )(\Delta (h))=h\ot 1,
\mbox{${\;\;\;}$}
(\va \ot \id)(\Delta (h))=1\ot h,\label{q2}
\end{eqnarray}
for all $h\in A$. The associator
$\Phi$ has to be a $3$-cocycle, in the sense that
\begin{eqnarray}
&&(1\ot \Phi)(\id\ot \Delta \ot \id)
(\Phi)(\Phi \ot 1)=
(\id\ot \id \ot \Delta )(\Phi )
(\Delta \ot \id \ot
\id)(\Phi ),\label{q3}\\
&&(\id \ot \va \ot \id)(\Phi )=1\ot 1\ot 1.\label{q4}
\end{eqnarray}

$A$ is called a quasi-Hopf
algebra if, moreover, there exists an
anti-morphism $S$ of the algebra
$A$ and elements $\a , \b \in
A$ such that, for all $h\in A$, we
have:
\begin{eqnarray}
&&S(h\_1)\a h\_2=\va(h)\a
\mbox{${\;\;\;}$ and ${\;\;\;}$}
h\_1\b S(h\_2)=\va (h)\b,\label{q5}\\
&&\Phi^1\b S(\Phi^2)\a \Phi^3=1
\mbox{${\;\;\;}$ and${\;\;\;}$}
S(\Phi^{-1})\a \Phi^{-2}\b S(\Phi^{-3})=1.\label{q6}
\end{eqnarray}
Here we use the notation $\Phi=\Phi^1\ot \Phi^2\ot \Phi^3$, $\Phi^{-1}= \Phi^{-1}\ot \Phi^{-3}\ot \Phi^{-3}$. If $A$ is a quasi-Hopf algebra, we shall denote by $\Rep(A)$ the tensor category of finite-dimensional representations of $A$.

\medbreak

An invertible element
$J\in A\ot A$ is called a  {\sl
twist} if $(\va \ot \id)(J)=1=(\id\ot \va)(J)$. If $A$ is a quasi-Hopf algebra and $J=J^1\ot J^2\in A\ot A$ is a twist with inverse $J^{-1}=J^{-1}\ot J^{-2}$, then we can
define a  quasi-Hopf algebra on the same algebra $A$
keeping the
 counit and antipode and replacing the
comultiplication, associator and the elements $\alpha$ and $\beta$ by
\begin{eqnarray}
&&\Delta_J(h)=J\Delta (h)J^{-1},\label{g1}\\
&&\Phi_J=(1\ot J)(\id \ot \Delta )(J) \Phi (\Delta \ot \id)
(J^{-1})(J^{-1}\ot 1),\label{g2}\\
&&\a_J=S(J^{-1})\a J^{-2},%
\mbox{${\;\;\;}$}%
\b_J=J^1\b S(J^2).\label{g3}
\end{eqnarray}
We shall denote this new quasi-Hopf algebra by $(A_J, \Phi_J)$. If $\Phi=1$ then, in this case, we shall denote $\Phi_J=dJ$.

\subsection{Comodule algebras over  quasi-Hopf algebras}

Let $(A, \Phi, \alpha,\beta,1)$ be a finite dimensional quasi-Hopf algebra.
\begin{defi} A left $A$-comodule algebra  is a family $(\kc, \lambda, \Phi_\lambda)$ such that $\kc$ is an algebra, $\lambda:\kc\to A\ot \kc$ is an algebra map, $\Phi_\lambda\in A\ot A\ot \kc$ is an invertible element such that
 \begin{align}\label{comod-alg1} &(1\ot \Phi_\lambda)(\id\otimes \Delta\ot\id)(\Phi_\lambda)(\Phi\ot 1)=(\id\ot\id\ot \lambda)(\Phi_\lambda)(\Delta\ot\id\ot \id) (\Phi_\lambda),
  \end{align}
 \begin{align}\label{comod-alg2}
      (\id\ot\epsilon\ot \id)(\Phi_\lambda)=1,
  \end{align}
  \begin{align}\label{comod-alg3}
       \Phi_\lambda(\Delta\ot\id)\lambda(x)=\big((\id\otimes \lambda)\lambda(x)\big)\Phi_\lambda, \quad x\in \kc
 \end{align}

We shall say that a comodule algebra $(\kc, \lambda, \Phi_\lambda)$ is right $A$-simple if it has no non-trivial right ideals $J\subseteq \kc$ such that $J$ is costable, that is $\lambda(J)\subseteq A\ot\kc$.
\end{defi}

\begin{rmk} The notion of comodule algebra for quasi-Hopf algebras does not coincide with the notion of comodule algebra for (usual) Hopf algebras. For quasi-Hopf algebras the coaction may not be coassociative.

\end{rmk}

If $(\kc, \lambda, \Phi_\lambda)$ is a left $A$-comodule algebra, the category ${}^{A}_{\kc}\Mo_A$ consists of $(\kc, A)$-bimodules  $M$ equipped with a $(\kc, A)$-bimodule map $\delta:M\to A\ot M$ such that for all $m\in M$
 \begin{align}\label{quasi-bimodule1}   \Phi_\lambda (\Delta\ot\id)\delta(m)&= (\id\ot\delta) \delta(m)\Phi,\\
 \label{quasi-bimodule2} (\varepsilon\ot\id)\delta&=\id.
\end{align}

The following result will be useful to present examples of exact module categories, it is a consequence of some freeness results on comodule algebras over quasi-Hopf algebras proven by H. Henker.

\begin{lema}\label{skryabin-quasi} Let $(\kc, \lambda, \Phi_\lambda)$ be a right $A$-simple left $A$-comodule algebra. If $M\in {}_\kc\Mo$ then $A\ot M\in {}_\kc\Mo$ is projective.
\end{lema}
\pf The object $A\ot M$ is in the category ${}^{A}_{\kc}\Mo_A$ as follows. The left $\kc$-action and the right $A$-action on $A\ot M$ are determined by
$$x\cdot (a\ot m)= x\_{-1}a\ot x\_0\cdot m, \quad (a\ot m)\cdot b= ab\ot m,$$
for all $x\in \kc$, $a,b\in A$ and $m\in M$. The coaction is determined by $\delta: A\ot M\to A\ot A\ot M$, $\delta=\Phi_\lambda(\Delta\ot \id_M)$.
It follows from \cite[Lemma 3.6]{He} that $A\ot M$ is a projective  $\kc$-module.
\epf

\subsection{Comodule algebras over radically graded quasi-Hopf algebras}
 Let $A$ be a  quasi-Hopf algebra \emph{radically graded}, that is there is an algebra grading
 $A=\oplus^m_{i=0} A[i]$, where $I:= \Rad A = \oplus_{i \geq 1} A[i]$ and  $I^k=  \oplus_{i \geq k} A[i]$
 for any $k=0\dots m$. Here $I^0=A$. Since $\Delta(I)\subseteq I\ot A + A\ot I$ then 
  $\Delta(I)\subseteq \sum_{j=0}^k\, I^j\ot I^{k-j}$ for any  $k=0\dots m$. In this case $A[0]$ is semisimple, $A$ is generated by $A[0]$ and $A[1]$, and the associator $\Phi$ is an element in $ A[0]^{\ot 3}$, see \cite[Lemma 2.1]{EG1}.

If $(\kc, \lambda, \Phi_\lambda)$ is a left $A$-comodule algebra, define
$$ \kc_i=\lambda^{-1}( I^i \ot \kc),\quad i=0\dots m.$$
This is an algebra filtration, thus we can consider the associated graded algebra $\gr \kc=\oplus_{i=0}^m \, \kc[i],\, \kc[i]=\kc_i/ \kc_{i+1}$.

\begin{lema}\label{graded-comod-alg}
\begin{enumerate}
\item[1.] The above filtration satisfies
\begin{align}\label{filtr-comod} \lambda(\kc_i)\subseteq \sum_{j=0}^i \, I^j \ot \kc_{i-j}.
\end{align}

\item[2.] There is a left $A$-comodule algebra structure $(\gr \kc,\overline{\lambda}, \overline{\Phi}_\lambda)$  satisfying
    \begin{align}\label{filtr-comod2}  \overline{\lambda}(\gr \kc(n))\subseteq \oplus_{k=0}^n \;A[k]\ot \kc[n-k].
\end{align}

\item[3.] $(\kc[0], \overline{\lambda}, \overline{\Phi}_\lambda) $ is a left $A[0]$-comodule algebra.
\end{enumerate}

\end{lema}

\pf Item (1) follows from the definition of $\kc_i$ and equation \eqref{comod-alg3}.  For each $n=0\dots m$ there is a linear map $\overline{\lambda}: \gr \kc\to  A\ot \gr \kc$ such that the following diagram commutes

$$
\begin{CD}
 \kc_n@>{\lambda}>>   \sum_{j=0}^n \, I^j \ot \kc_{n-j}
\\
@V{\pi}VV @VVV \\
 \kc_n/ \kc_{n+1} @>>> \big(\sum_{j=0}^n \, I^j \ot \kc_{n-j}\big)/ \sum^{n+1}_{j=0}\, I^j \ot \kc_{n+1-j}
 \\
@VVV @VV\simeq V \\
\kc[n] @>{\overline{\lambda}}>> \oplus^n_{k=0}\, A[k] \otk\,
\kc[n-k].
\end{CD}
$$
Defining $\overline{\Phi}_\lambda$ as the projection of $\Phi_\lambda$ to $A[0]\ot A[0]\ot \kc[0]$ follows immediately that $(\gr \kc,\overline{\lambda}, \overline{\Phi}_\lambda)$ is a left $A$-comodule algebra.
\epf

\begin{lema}\label{right-simplicity} The following statements are equivalent:
\begin{enumerate}
  \item[1.] $\kc$ is a right $A$-simple left $A$-comodule algebra.
  \item[2.] $\kc[0]$ is a right $A[0]$-simple left $A[0]$-comodule algebra.
  \item[3.] $\gr\kc$ is a right $A$-simple left $A$-comodule algebra.
\end{enumerate}
\end{lema}

\pf Assume $\kc[0]$ is a right $A[0]$-simple. Let $J\subseteq A$ be a right ideal $A$-costable. Consider the filtration $J=J_0 \supseteq J_1 \supseteq\dots \supseteq J_m$ given by $J_k=\lambda^{-1}(I^k\ot J)$ for all $k=0\dots m$. Set $\overline{J}(k)= J_k/J_{k+1}$ for any $k$ and $\overline{J}=\oplus_k\, \overline{J}(k)$. It follows that for any $n=0\dots m$
\begin{align}\label{filtr:on-ideals}  \overline{\lambda}(\overline{J}(n))\subseteq \oplus_{k=0}^n \;A[k]\ot \overline{J}(k).
\end{align}
In particular $\overline{J}(0)\subseteq \kc[0]$ is a right ideal $A[0]$-costable thus $\overline{J}=\kc[0]$ or $\overline{J}=0$. In the first  case $J=A$ and in the second case $J=J_1$. It follows from \eqref{filtr:on-ideals} that $\overline{J}(1)\subseteq\kc[0]$ is a a right ideal $A[0]$-costable. Hence $J=J_2$. Continuing this reasoning we obtain that $J=0$.

\medbreak

Assume now that $ \kc$ is a right $A$-simple. Let $\overline{J}\subseteq \kc[0]$ be a right  $A[0]$-costable ideal. Denote $\pi:\kc\to \kc[0]$ the canonical projection and $J=\pi^{-1}(\overline{J})$. Clearly $J$ is a right $A$-costable  ideal thus $J=0$ or $J=\kc$, thus $\overline{J}=0$ or $\overline{J}=\kc[0]$ respectively.\epf

As a consequence we have the following result.

\begin{cor} Let $(A,\Phi)$ be a  radically graded quasi-Hopf algebra and $(\kc, \lambda, \Phi_\lambda)$ be a left $A$-comodule algebra such that $\kc[0]=\ku 1$ then $A$ is twist equivalent to a Hopf algebra.
\end{cor}
\pf Since $(\kc[0], \overline{\lambda}, \overline{\Phi}_\lambda) $ is a left $A[0]$-comodule algebra then there exists an invertible element $J\in A\ot A$ such that $J\ot 1=\overline{\Phi}_\lambda$. Equation \eqref{comod-alg1} implies that $\Phi=dJ$.\epf

\subsection{Module categories over quasi-Hopf algebras}
For any comodule algebra over a quasi-Hopf algebra $A$ there is associated a module category over $\Rep(A)$.

\begin{lema}\label{exact-mod-quasi} Let $A$ be a finite-dimensional quasi-Hopf algebra.
\begin{enumerate}
  \item[1.] If $(\kc, \lambda, \Phi_\lambda)$ is a left $A$-comodule algebra then the category ${}_{\kc}\Mo$ is a module category over $\Rep(A)$. It is exact if $\kc$ is right $A$-simple.
 \item[2.] If $\Mo$ is an exact module category over $\Rep(A)$ there exists a left $A$-comodule algebra $(\kc, \lambda,\Phi_\lambda)$ such that $\Mo\simeq {}_{\kc}\Mo$ as module categories over $\Rep(A)$.
      \end{enumerate}
\end{lema}
\pf 1. The action $\otb:\Rep(A)\times {}_{\kc}\Mo\to {}_{\kc}\Mo$ is given by the tensor product over the field $\ku$ where the action on the tensor product is given by $\lambda$. The associativity isomorphisms $m_{X,Y,M}: (X\otimes Y)\otimes M \to X\otimes (Y\otimes M)$ are given by
$$ m_{X,Y,M}(x\ot y\ot m)=\Phi^1_\lambda\cdot x\ot \Phi^2_\lambda\cdot y\ot \Phi^3_\lambda\cdot m,$$
for all $x\in X$, $y\in Y$, $M\in M$, $X, Y\in \Rep(A)$, $M\in {}_{\kc}\Mo$. To prove that ${}_{\kc}\Mo$ is exact, it is enough to verify that $A\ot M$ is projective for any $M\in {}_{\kc}\Mo$ but this is Lemma \ref{skryabin-quasi}.

2. This is a straightforward consequence of \cite[Thm. 3.17]{eo}, the proof of \cite[Prop. 1.19]{AM} extends \textit{mutatis mutandis} to the quasi-Hopf setting.
\epf

\begin{defi} Two left $A$-comodule algebras $(\kc, \lambda, \Phi_\lambda)$, $(\kc', \lambda', \Phi'_\lambda)$ are \textit{equivariantly Morita equivalent} if the corresponding module categories are equivalent.
\end{defi}

\subsection{Comodule algebras coming from twisting}\label{twisting:comodulealg}

Let $(A, \Phi)$ be a quasi-Hopf algebra and $J\in A\ot A$ be a twist. Let $(K, \lambda, \Phi_\lambda)$ be a left $A$-comodule algebra. Let us denote by $(K_J, \lambda_J, \widetilde{\Phi}_\lambda)$ the following left $A_J$-comodule algebra. As algebras $K_J=K$, the coaction $\lambda_J=\lambda$ and $\widetilde{\Phi}_\lambda=\Phi_\lambda(J^{-1}\ot 1)$.

The following results are straightforward.
\begin{lema}\label{comod-alg:twisting} $(K_J, \lambda_J, \widetilde{\Phi}_\lambda)$ is a left
  $A_J$-comodule algebra. It is right $A$-simple if and only if $(K, \lambda, \Phi_\lambda)$ is right $A$-simple.\qed
\end{lema}

\begin{lema}\label{morita-by-twisting} Let $J\in A\ot A$ be a twist. If $(K, \lambda, \Phi_\lambda)$ and $(K', \lambda', \Phi'_\lambda)$ are equivariantly Morita equivalent  $A$-comodule algebras then $(K_J, \lambda_J, \Phi_\lambda (J^{-1}\ot 1))$ and $(K'_J, \lambda'_J, \Phi'_\lambda (J^{-1}\ot 1))$ are equivariant Morita equivalent  $A_J$-comodule algebras.\qed
\end{lema}

\section{Equivariantization of quasi-Hopf algebras}\label{s:equivarin-quasi}

For a quasi-Hopf algebra $A$ we shall explain the notion of a \textit{crossed system over} $A$ and discuss its relation with  the equivariantization of the category $\Rep(A)$.

\smallbreak

Let $A_1, A_2$ be quasi-Hopf algebras. A {\em twisted homomorphism} between $A_1$ and $A_2$ is pair
$(f,J)$ consisting of a homomorphism of algebras $f:A_1\to A_2$
and an invertible element $J\in A_{2}^{\otimes 2}$ such that
\begin{equation}\label{tw1}
\Phi_2(\Delta\otimes \id)(J)(J\otimes 1) =
 (\id\otimes\Delta )(J)(1\otimes J)(f^{\otimes 3})(\Phi_1 ),
\end{equation}
\begin{equation}\label{tw2}
(\va \ot \id)(J)=(\id\ot \va)(J)=1,
\end{equation}
\begin{equation}\label{tw3}
\varepsilon(f(a))= \varepsilon(a),
\end{equation}

\begin{equation}\label{tw4}
\Delta (f(a))J = J(f^{2\otimes}(\Delta (a))), \quad \text{ for all }
a\in A.
\end{equation}

\begin{rmk}
If $(f,J):A_1\to A_2$ is a twisted homomorphism,  then $J^{-1}\in A_2\otimes A_2$ is a twist and  $f:A_1\to (A_2)_{J^{-1}}$ is a homomorphism of quasi-bialgebras.
\end{rmk}

We define the category $\underline{\End}^{\Tw}(A_1,A_2)$ whose objects  are twisted homomorphism from $A_1$ to $A_2$. A {\em morphism}  between two twisted homomorphisms
$(f,J),(f',J'):A_1\to A_2$ is an element $c\in A_2$ such that $cf(a) = f'(a)c$
for any $a\in A_1$ and $\Delta(c)J = J'(c\otimes c)$. The composition of $a:f\to g$, $b:g\to h$, is  $ba: f\to h$. If $(f,J_f):A_1\to A_2$ and $(g,J_g):A_2\to A_3$ are twisted homomorphism, we define the composition as the twisted homomorpshism $(g \circ f, J_g (g\otimes g)(J_f)):A_1\to A_3$.

\medbreak

To any twisted homomorphism $(f, J):A_1\to A_2$ there is associated  a tensor functor $$(f^* , \xi^J):\Rep(A_2)\to \Rep(A_1),$$
where $f^*(V)=V$ for all $V\in \Rep(A_2)$, and $f^*$ is the identity over arrows. The $A_1$-action on $f^*(V)$ is given through the morphism $f$.  The monoidal structure is given by applying the element
$J\in A_{2}^{\otimes 2}$:
$${\xi^J}_{M,N}:f^* (M)\otimes f^* (N) \to f^* (M\otimes N),\quad
{\xi^J}_{M,N}(m\otimes n) = J(m\otimes n),$$
for any $M, N\in \Rep(A_2)$, $m\in M$, $n\in N$.
Morphisms between twisted homomorphisms $f,f':A_1\to A_2$
of quasi-Hopf algebras correspond to tensor natural transformations
between the associated tensor functors.

\subsection{Crossed system over a quasi-Hopf algebra}
Given a quasi-Hopf algebra $A$ we shall denote by $\underline{\Aut}^{\Tw}(A)$ the (monodial) subcategory of $\underline{\End}^{\Tw}(A)$ where objects are twisted automorphisms of $A$, and arrows are isomorphisms of twisted automorphisms.
\smallbreak
Let $G$ be a group, and let $A$ be a quasi-Hopf algebra. A \emph{$G$-crossed system} over $A$ is a monoidal functor $*: \underline{G}\to \underline{\Aut}^{\Tw}(A)$ such that $e_*= (\id_A, 1\otimes 1)$.
\smallbreak

More explicitly a $G$-crossed system consists of the following data:

\begin{itemize}
  \item A twisted automorphism $(\sigma_*, J_\sigma)$ for each $\sigma\in G$,
  \item an element $\theta_{(\sigma,\tau)}\in A^{\times}$ for each $\sigma, \tau \in G$,
\end{itemize}
such that for all $a\in A$, $\sigma,\tau, \rho \in G$,

\begin{align}
\label{(i)} \varepsilon(\theta_{(\sigma,\tau)})&=1,\\
 \label{(ii)} (1_*, J_1)&=(\id, 1\ot 1),\\
  \label{(iii)} \theta_{(\sigma,\tau)}(\sigma\tau)_*(a)&= \sigma_*(\tau_*(a))\theta_{(\sigma,\tau)},\\
  \label{(iv)} \theta_{(\sigma,\tau)}\theta_{(\sigma\tau,\rho)} &= \sigma_*(\theta_{(\tau,\rho)})\theta_{\sigma,\tau\rho},\\
  \label{(v)} \theta_{(1,\sigma)}&=\theta_{(\sigma,1)}=1,\\
  \label{(vi)} \Delta(\theta_{(\sigma,\tau)})J_{\sigma\tau} &= J_\sigma (\sigma_* \otimes \sigma_*)(J_\tau)(\theta_{(\sigma,\tau)}\otimes\theta_{(\sigma,\tau)} ).
\end{align}

\medbreak

Let $A\# G$ be the vector space $A\otk \ku G$ with product and coproduct $$(x\#\sigma)(y\# \tau)= x\sigma_*(y)\theta_{(\sigma,\tau)}\#\sigma\tau,\quad\Delta(x\#\sigma)= x\_1J^1_\sigma\# \sigma\otimes x\_2J^2_\sigma\#\sigma,$$
for all $x,y\in A$, $\sigma, \tau\in G$.

\begin{prop}\label{prop producto cruzado}
The foregoing operations makes the vector space $A\# G$ into a quasi-bialgebra with associator $\Phi^1\#e\otimes \Phi^2\#e\otimes \Phi^3\#e$, and counit $\varepsilon(x\# \sigma)=\varepsilon(x)$ for all $x\in A, \sigma\in G$.
\end{prop}
\begin{proof}
It is straightforward to see that $A\# G$ is an associative algebra
with unit $1\#e$. Equation \eqref{q1} follows from  \eqref{tw1}. The
map $\varepsilon$ is an algebra morphism by \eqref{tw3} and
\eqref{(i)}. Equations \eqref{q2} follow from \eqref{tw2}, equations
\eqref{q3} and \eqref{q4} follow by the definition of the
associator. Finally $\Delta$ is an algebra morphism by \eqref{tw4}
and \eqref{(vi)}.
\end{proof}
\subsection{Antipodes of crossed systems}

Let $G$ be a group, $(A,\Phi, S,\alpha,\beta)$ be a quasi-Hopf algebra and $(\sigma_*,\theta_{(\sigma,\tau)}, J_\sigma)_{\sigma,\tau \in G}$ a $G$-crossed system over $A$. An antipode for $(\sigma_*,\theta_{(\sigma,\tau)}, J_\sigma)_{\sigma,\tau \in G}$ is a function $\upsilon: G\to A^{\times}$ such that

\begin{gather}
    \upsilon_{\sigma\tau}(\sigma\tau)_*(S(\theta_{\sigma,\tau})) =\upsilon_\tau(\tau^{-1})_*(\upsilon_\sigma)\theta_{\tau^{-1},\sigma^{-1}},\\
    \upsilon_\sigma^{-1}S(x)\upsilon_\sigma =  (\sigma^{-1})_*(S(\sigma_*(x))),\\
    \upsilon_\sigma(\sigma^{-1})_*((S(J_\sigma^1)\alpha J_\sigma^2))\theta_{\sigma^{-1},\sigma}=\alpha,\\
    J_\sigma^{1}\sigma_*(\beta\upsilon_\sigma(\sigma^{-1})_*(S(J^2_\sigma))) \theta_{\sigma,\sigma^{-1}}=\beta,
\end{gather}for all $\sigma \in G$, where $J_\sigma= J_\sigma^1\otimes J_\sigma^2$. The next proposition follows by a straightforward verification.

\begin{prop}
Let $\upsilon: G\to A^{\times}$ be an antipode for $(\sigma_*,\theta_{(\sigma,\tau)}, J_\sigma)_{\sigma,\tau \in G}$. Then $(S,\alpha\# e, \beta\#e)$ is an antipode for $A\# G$, where
$$S(x\# \sigma)= \upsilon_\sigma(\sigma^{-1})_*(S(x))\#\sigma^{-1},$$ for all $\sigma \in G, x\in A$.\qed\end{prop}
\subsection{Equivariantization and crossed systems}
Let us assume that $G$ is an Abelian group. In this case a $G$-crossed system over $A$ gives rise to a $G$-action on the category $\Rep(A)$. Indeed, for any $\sigma\in G$ we can define the tensor functors $(F_\sigma, \zeta_\sigma): \Rep(A)\to \Rep(A)$ described as follows. For any $V\in \Rep(A)$, $F_\sigma(V)=V$ as vector spaces and the action on $F_\sigma(V)$ is given by $a\cdot v=\sigma_*(a)v$ for all $a\in A$, $v\in V$.
For any $V, W\in \Rep(A)$ the isomorphisms $(\zeta_\sigma)_{V,W}:V\ot W\to V\ot W$ are given by $(\zeta_\sigma)_{V,W}(v\ot w)=J_\sigma\cdot (v\ot w)$ for all $v\in V$, $w\in W$.
For any $\sigma, \tau\in G$ the natural tensor transformation $\gamma_{\sigma, \tau}:F_\sigma\circ F_\tau\to F_{\sigma\tau}$, $(\gamma_{\sigma, \tau})_V(v)=\theta^{-1}_{(\sigma,\tau)}v$ for all $V\in \Rep(A)$, $v\in V$.

\begin{lema}\label{g-crossed-action} If $\theta_{(\sigma,\tau)}=\theta_{(\tau,\sigma)}$ for all $\sigma, \tau\in G$
 then the tensor functors $(F_\sigma, \zeta_\sigma)$ described above define a $G$-action on $\Rep(A)$.
\end{lema}
\pf The conmutativity of $G$ and equation $\theta_{(\sigma,\tau)}=\theta_{(\tau,\sigma)}$ for all $\sigma, \tau\in G$
imply that the maps $\gamma_{\sigma, \tau}$ are morphisms of $A$-modules.  The proof that  the tensor functors $(F_\sigma, \zeta_\sigma)$
define a $G$-action is straightforward.
\epf

Given  a  $G$-crossed system $(\sigma_*,\theta_{(\sigma,\tau)}, J_\sigma)_{\sigma,\tau \in G}$
 over $A$ we consider the category $\Rep(A)^{G}$ of $G$-equivariant $A$-modules.

\begin{prop}\label{equi}
Let $G$ be an Abelian group, $A$ be a quasi-Hopf algebra and $(\sigma_*,\theta_{(\sigma,\tau)}, J_\sigma)_{\sigma,\tau \in G}$ a $G$-crossed system over $A$ such that $\theta_{(\sigma,\tau)}=\theta_{(\tau,\sigma)}$ for all $\sigma,\tau\in G$. Then there is a tensor equivalence between $\Rep(A)^{G}$ and $\Rep(A\# G)$.
\end{prop}
\begin{proof}
Let $(V,u)$ be a $G$-equivariant object. The linear isomorphisms $u_\sigma:F_\sigma(V)\to V$ satisfy
\begin{align}
 u_\sigma(\sigma_*(a)\cdot v)=a\cdot u_\sigma(v), \quad u_\sigma(u_\tau(v)) &=u_{\sigma\tau}(\theta_{(\sigma,\tau)}\cdot v)\label{1equi-V}
\end{align}
for all $v\in V$, $a\in A$, $\sigma,\tau \in G$. Equation \eqref{1equi-V} together with the fact that $\theta_{(\sigma,\tau)}=\theta_{(\tau,\sigma)}$ for all $\sigma,\tau\in G$ imply that
there is a well-defined action of the crossed product $A\# G$ on $V$ determined
by
\begin{equation}
(a\# \sigma)\cdot v =a u^{-1}_\sigma(v),
\end{equation}
for all $a\in A, v\in V, \sigma \in G$. Morphisms of $G$-equivariant representations are exactly  morphisms of  $A\# G$-modules. Hence we have defined a functor $$\mathcal{F}:\Rep(A)^G \to\Rep(A\#G),$$
which clearly is a tensor functor. Assume that $W \in \Rep(A\#G)$. Then, by
restriction, $W$ is a representation of $A$. Moreover $(W, u)$ is a $G$-equivariant object in $\Rep(A)$, letting
$$u_\sigma : W \to W,
\quad u_\sigma (w) = (\theta^{-1}_{(\sigma,\sigma^{-1})}\# \sigma^{-1})\cdot w,$$ for every $\sigma \in G$. We have
thus a functor $\mathcal{G} : \Rep(A\#G) \to \Rep(A)^G$. It is clear that $\mathcal{F}$ and $\mathcal{G}$ are inverse equivalences of categories.
\end{proof}

\begin{rmk} A version of the above result appears in \cite[Prop. 3.2]{Na}.
\end{rmk}

\subsection{Crossed product of quasi-bialgebras}
\begin{defi}\label{defi:crossedp}
Let $(A,\Phi, S,\alpha,\beta)$ be a quasi-Hopf algebra, and let $G$ be a group. We shall say that $A$ is a $G$-\emph{crossed product} if there is a decomposition $A=\bigoplus_{\sigma\in G} A_\sigma$, where:
\begin{itemize}
\item $\Phi\in A_e\otimes A_e\otimes A_e$,
  \item $A_\sigma A_\tau \subseteq A_{\sigma\tau}$ for all $\sigma,\tau \in G$,
  \item $A_\sigma$ has an invertible element for each $\sigma \in G$,
  \item $\Delta(A_\sigma)\subseteq A_\sigma\otimes A_\sigma$ for each $\sigma \in G$.

   \item  $S(A_\sigma)\subseteq A_{\sigma^{-1}}$,  for each $\sigma \in G$.
  \item $\alpha, \beta \in A_e$
\end{itemize}
\end{defi}

\begin{prop}\label{graduado implica sistema cruzado}
Every $G$-crossed product $A$ is of the form $ B\# G$ for some quasi-Hopf algebra $B$. Moreover, there exists an antipode $\upsilon: G\to B^{\times}$ such that $ B\# G$ is isomorphic to $A$ as quasi-Hopf algebras.
\end{prop}

\begin{proof}
Let $A$ be a $G$-crossed product. Set $B=A_e$. Since every $A_\sigma$ has an invertible element, we may choose for each $\sigma \in G$ some invertible element $t_\sigma\in A_\sigma$, with $t_e=1$. Then it is clear that $A_\sigma =t_\sigma A_e= A_e t_\sigma$, and the set $\{t_\sigma: \sigma\in G\}$ is a basis for $A$ as a left (and  right) $A_e$-module. Note that $\varepsilon(t_\sigma)\neq 0$, because $\varepsilon$ is an algebra map and $t_\sigma$ is invertible. Thus, we may and shall assume that $\varepsilon(t_\sigma)=1$ for each $\sigma \in G$. Let us define the maps $$\sigma_*(a)=t_\sigma at_\sigma^{-1}, \text{ for each $\sigma\ G$ and $a\in A_e$,}$$ and $$\theta:G\times G\to A \;\; \text{ by } \; \theta_{(\sigma,\tau)}=t_\sigma t_\tau t_{\sigma\tau}^{-1}\; \text{ for } \sigma, \tau \in G.$$

We have that $\Delta(t_\sigma)\in A_\sigma\otimes A_\sigma$ can be uniquely expressed as $\Delta(t_\sigma)= J_\sigma(t_\sigma\otimes t_\sigma)$, with $J_\sigma \in A_e\otimes A_e$. Since $\Delta$ is an algebra morphism, $J_\sigma$ is invertible, and for the normalization $\varepsilon(t_\sigma)=1$,  $(\varepsilon\otimes\id) (J_\sigma)= (\id\otimes \varepsilon)(J_\sigma)=1$.

Then, it is straightforward to see that the data $(\sigma_*,\theta_{(\sigma,\tau)}, J_\sigma)_{\sigma,\tau \in G}$, defines a $G$-crossed system over the sub-quasi-bialgebra  $A_e\subseteq A$, and $A_e\# G$ is isomorphic to $A$ as quasi-bialgebras.

The antipode $S: A\to A$ is anti-isomorphism of algebras, and the condition $S(A_\sigma)\subset A_{\sigma^{-1}}$ implies that there is a unique function  $\upsilon: G\to  A_e^{\times}$ such that $S(t_\sigma)=\theta_\sigma t_{\sigma^{-1}}$ for all $\sigma \in G$. Hence, it is straightforward to see that $\upsilon$ is antipode for the crossed system $(\sigma_*,\theta_{(\sigma,\tau)}, J_\sigma)_{\sigma,\tau \in G}$, and $A_e\# G$ is isomorphic to $A$ as quasi-Hopf algebras.
\end{proof}

\subsection{Twisted homomorphisms of comodule algebras}

Let $A$ be a quasi-Hopf algebra. A {\em twisted homomorphism} of left $A$-comodule algebras  $(\kc, \l, \Phi_\l)$ and $(\kc', \l', \Phi_\l')$ is pair
$(\mathfrak{f},\mathfrak{J})$ consisting of a homomorphism of algebras $\fs:\kc\to \kc'$
and an invertible element $\js\in A\otimes \kc'$ such that
\begin{equation}\label{tw-com-alg1}
 \Phi_{\lambda'}(\Delta\otimes \id)(\js)= (\id\otimes\lambda')(\js)(1\otimes J)(\id\otimes\id\otimes \fs)(\Phi_\lambda),
\end{equation}
\begin{equation}\label{tw-com-alg2}
(\va \ot \id)(\js)=1,
\end{equation}
\begin{equation}
\lambda' (\fs(a))\js = \js(\id\otimes \fs)(\lambda (a)), \quad \text{ for all } a\in \kc.
\end{equation}

A {\em morphism} between two twisted homomorphisms
$(\fs_1,\js_1),(\fs_2,\js_2):\kc\to \kc'$ is an element $c\in \kc'$ such that $c\fs_1(a) = \fs_2(a)c$
for any $a\in \kc$ and $\lambda'(c)\js_1 = \js_2(1\otimes c)$.

To any twisted homomorphism of comodule algebras $(\fs, \js):\kc\to \kc'$ there is associated  a $\Rep(A)$-module
functor
$$(\fs^* , \xi^\js):\Rep(\kc')\to \Rep(\kc),$$
where, for all $V\in \Rep(\kc_2)$, $\fs^*(V)=V$ with action given by $x\cdot v=\fs(x) v$, $x\in \kc$, $v\in V$. The natural transformation $\xi^\js$ is given by
$${\xi^\js}_{X,M}:  \fs^* (X \otimes M)\to X\otimes \fs^*(M),\quad
{\xi^\js}_{X,M}(x\otimes m) = \js^{-1}\cdot(x\otimes m),$$
for any $X \in \Rep(A), M\in \Rep(\kc_2)$, $x\in X$, $m\in M$.
Morphisms between twisted homomorphisms $\fs,\fs':\kc\to \kc'$
of $A$-comodule algebras correspond to module natural transformations
between the module functors.

\medbreak

Let $A$ be a quasi-Hopf algebra and $(\kc, \l, \Phi_\l)$ be a left $A$-comodule algebra. For each twisted endomorphism $(f,J):A\to A$, we define a new left $A$-comodule algebra $(\kc^f, \l^{f}, \Phi_\l^{f})$, where $\kc^f = \kc$ as algebras and
\begin{align*}
    \l^{f}(x)=(f\otimes\id)\l(x),\quad
    \Phi_\l^f = (f\otimes f\otimes\id) (\Phi_\l)(J^{-1}\otimes 1),
\end{align*}
for all $x\in \kc$.

\begin{defi}
Let $A$ be a quasi-Hopf algebra and  $(\kc, \l, \Phi_\l)$ be a left $A$-comodule algebra. Given a twisted endomorphism $(f,J)$ of $A$, a $(f,J)$-\emph{twisted endomorphism of} $\kc$  is a  twisted homomorphism from  $(\kc^f, \l^{f}, \Phi_\l^{f})$ to $(\kc, \l, \Phi_\l)$. Explicitly a $(f,J)$-twisted endomorphism is a pair $(\fs,\js)$ consisting of an algebra endomorphism $\fs:\kc \to \kc$ and an invertible element $\js\in A\otimes \kc$, such that:
\begin{equation}\label{twisted-end-comodalg1}(\varepsilon\otimes \id)(\js)=1,
   \end{equation}
   \begin{equation}\label{twisted-end-comodalg2}\Phi_\l(\Delta\otimes \id)(\js)(J\otimes 1)= (\id\otimes \l)(\js)(1\otimes \js)(f\otimes f\otimes \fs)(\Phi_\l)
\end{equation}
\begin{equation}\label{twisted-end-comodalg3}\l(\fs(x))\js=\js(f\otimes \fs)(\l(x)),\quad \text{ for all } x\in \kc.
\end{equation}
\end{defi}

\begin{lema}\label{lema composicion}
Let $(\sigma_*, \theta_{(\sigma,\tau)}, J_\sigma)_{\sigma,\tau \in G}$ be a crossed system over a quasi-Hopf algebra $A$, and $(\kc,\lambda,\Phi_\l)$ a left $A$-comodule algebra. If $(\fs_\sigma, \js_\sigma), (\fs_\tau,\js_{\tau}): \kc\to \kc$ are $(\sigma_*,J_\sigma)$-twisted and $(\tau_*,J_\tau)$-twisted endomorphism, then $$(\fs_\sigma, \js_\sigma)\overline{\circ}(\fs_\tau,\js_{\tau})=(\fs_\sigma\circ \fs_\tau ,\js_\sigma (\sigma_*\otimes \fs_\sigma)(\js_\tau)(\theta_{\sigma,\tau}\otimes 1 ))$$ is a  $((\sigma\tau)_*,J_{\sigma\tau})$-twisted endomorphism. Moreover, this composition is associative, \textit{i.e.}, if $(\fs_\sigma, \js_\sigma), (\fs_\tau,\js_{\tau}), (\fs_\rho,\js_{\rho}): \kc\to \kc$ are $(\sigma_*,J_\sigma)$-twisted, $(\tau_*,J_\tau)$-twisted, and $(\rho_*,J_\rho)$-twisted endomorphism, then $$[(\fs_\sigma, \js_\sigma)\overline{\circ}(\fs_\tau,\js_{\tau})]\overline{\circ} (\fs_\rho,\js_{\rho}) = (\fs_\sigma, \js_\sigma)\overline{\circ}[(\fs_\tau,\js_{\tau})\overline{\circ} (\fs_\rho,\js_{\rho})]$$
\end{lema}
\begin{proof}

If we use the following notation  \[\js_\sigma\overline{\circ}\ \js_\tau= \js_\sigma (\sigma_*\otimes \fs_\sigma)(\js_\tau)(\theta_{\sigma,\tau}\otimes 1 ),\] thus we need to prove:
\begin{enumerate}
  \item $(\varepsilon\otimes \id)(\js_\sigma\overline{\circ}\ \js_\tau)=1$,
  \item $\Phi_\l(\Delta\otimes \id)(\js_\sigma\overline{\circ}\ \js_\tau)(J_{\sigma\tau}\otimes 1)= (\id\otimes \l)(\js_\sigma\overline{\circ}\ \js_\tau)(1\otimes \js_\sigma\overline{\circ}\ \js_\tau)((\sigma\tau)_*\otimes (\sigma\tau)_*\otimes \fs_\sigma\circ \fs_\tau)(\Phi_\l)$,
  \item $\l(\fs_\sigma\circ \fs_\tau(x))\js_\sigma\overline{\circ}\ \js_\tau=\js_\sigma\overline{\circ}\ \js_\tau(f\otimes \fs)(\l(x))$ for all $x\in \kc$.
\end{enumerate}

(1)\ The first equation follows immediately using   $\varepsilon\otimes \id (\js_\sigma)=1$, and $\varepsilon $  is an algebra morphism that commutes with $\sigma_*$ for all $\sigma \in G$.

(2)\ For the second equation, first we shall see some equalities:
\begin{gather}
(\Delta\otimes \id) (\sigma_*\otimes \fs_\sigma)(\js_\tau)(J_\sigma\otimes 1)
=(J_\sigma\otimes 1)(\sigma_*\otimes\sigma_*\otimes \fs_\sigma)(\Delta\otimes \id)(\js_\tau)\label{eq 1}\\
(1\otimes\js_\sigma)(\sigma_*\otimes\sigma_*\otimes \fs_\sigma)[(\id\otimes \lambda )(\js_\tau)]= (\id\otimes\lambda) (\sigma_*\otimes \fs_\sigma)(\js_\tau)(1\otimes \js_\sigma)\label{eq 2}
\end{gather}

The equation \eqref{eq 1} follows by axiom \eqref{tw4} of
$J_{\sigma\tau}$, and the equation \eqref{eq 2} follows by axiom
\eqref{twisted-end-comodalg3} of $\js_\sigma$.
\begin{multline}\label{eq 3}
\Phi_\l(\Delta\otimes \id)(\js_\sigma (\sigma_*\otimes \fs_\sigma)(\js_\tau))(J_\sigma\otimes 1)\\
{\tiny\eqref{eq 1}}\  =\Phi_\l(\Delta\otimes \id)(\js_\sigma)(\Delta\otimes \id) (\sigma_*\otimes \fs_\sigma)(\js_\tau)(J_\sigma\otimes 1)\\
{\tiny\eqref{twisted-end-comodalg2}} =\Phi_\l(\Delta\otimes \id)(\js_\sigma)(J_\sigma\otimes 1)(\sigma_*\otimes\sigma_*\otimes \fs_\sigma)(\Delta\otimes \id)(\js_\tau)\\
=(\id\otimes \l)(\js_\sigma)(1\otimes \js_\sigma)(\sigma_*\otimes \sigma_*\otimes \fs_\sigma)(\Phi_\l)(\sigma_*\otimes\sigma_*\otimes \fs_\sigma)(\Delta\otimes \id)(\js_\tau)\\
=(\id\otimes \l)(\js_\sigma)(1\otimes \js_\sigma)(\sigma_*\otimes \sigma_*\otimes \fs_\sigma)[\Phi_\l(\Delta\otimes \id)(\js_\tau)]
\end{multline}

\begin{multline}\label{eq 4}
    \Phi_\l(\Delta\otimes \id)(\js_\sigma (\sigma_*\otimes \fs_\sigma)(\js_\tau))[J_\sigma (\sigma_* \otimes \sigma_*)(J_\tau)\otimes 1]\\
   {\tiny\eqref{eq 3}} = (\id\otimes \l)(\js_\sigma)(1\otimes \js_\sigma)(\sigma_*\otimes \sigma_*\otimes \fs_\sigma)[\Phi_\l(\Delta\otimes \id)(\js_\tau)]((\sigma_* \otimes \sigma_*)(J_\tau))\otimes 1\\
    = (\id\otimes \l)(\js_\sigma)(1\otimes \js_\sigma)(\sigma_*\otimes \sigma_*\otimes \fs_\sigma)[\Phi_\l(\Delta\otimes \id)(\js_\tau)(J_\tau\otimes 1)]\\
    {\tiny\eqref{twisted-end-comodalg2}} = (\id\otimes \l)(\js_\sigma)(1\otimes \js_\sigma)(\sigma_*\otimes \sigma_*\otimes \fs_\sigma)[(\id\otimes \l)(\js_\tau)(1\otimes \js_\tau)(\tau_*\otimes \tau_*\otimes \fs_\tau)(\Phi_\l)]\\
    = (\id\otimes \l)(\js_\sigma)(1\otimes \js_\sigma)(\sigma_*\otimes \sigma_*\otimes \fs_\sigma)[(\id\otimes \l)(\js_\tau)(1\otimes \js_\tau)](\sigma_*\tau_*\otimes \sigma_*\tau_*\otimes \fs_\sigma\fs_\tau)(\Phi_\l)\\
    = (\id\otimes \l)(\js_\sigma)(1\otimes \js_\sigma)(\sigma_*\otimes \sigma_*\otimes \fs_\sigma)[(\id\otimes \l)(\js_\tau)] \\
    \times [1\otimes(\sigma_*\otimes \fs_\sigma) (\js_\tau)](\sigma_*\tau_*\otimes \sigma_*\tau_*\otimes \fs_\sigma\fs_\tau)(\Phi_\l)\\
   {\tiny \eqref{eq 2}}   = (\id\otimes \l)(\js_\sigma)(\id\otimes\lambda) (\sigma_*\otimes \fs_\sigma)(\js_\tau)(1\otimes \js_\sigma) \\
    \times [1\otimes(\sigma_*\otimes \fs_\sigma) (\js_\tau)](\sigma_*\tau_*\otimes \sigma_*\tau_*\otimes \fs_\sigma\fs_\tau)(\Phi_\l)\\
    = (\id\otimes \l)[(\js_\sigma)(\sigma_*\otimes \fs_\sigma)(\js_\tau)](1\otimes \js_\sigma) \\
    \times [1\otimes(\sigma_*\otimes \fs_\sigma) (\js_\tau)](\sigma_*\tau_*\otimes \sigma_*\tau_*\otimes \fs_\sigma\fs_\tau)(\Phi_\l)\\
     = (\id\otimes \l)(\fs_\sigma\overline{\circ}\ \fs_\tau)(1\otimes \fs_\sigma\overline{\circ}\ \fs_\tau)(\theta_{(\sigma,\tau)}^{-1}\otimes\theta_{(\sigma,\tau)}^{-1}\otimes 1)(\sigma_*\tau_*\otimes \sigma_*\tau_*\otimes \fs_\sigma\fs_\tau)(\Phi_\l)\\
     {\tiny\eqref{(iii)}} = (\id\otimes \l)(\js_\sigma\overline{\circ}\ \js_\tau)(1\otimes \js_\sigma\overline{\circ}\ \js_\tau)((\sigma\tau)_*\otimes (\sigma\tau)_*\otimes \fs_\sigma\fs_\tau)(\Phi_\l)(\theta_{(\sigma,\tau)}^{-1}\otimes\theta_{(\sigma,\tau)}^{-1}\otimes 1)
\end{multline}


\begin{multline*}
\Phi_\l(\Delta\otimes \id)(\js_\sigma\overline{\circ}\ \js_\tau )(J_{\sigma\tau}\otimes 1)\\
   = \Phi_\l(\Delta\otimes \id)(\js_\sigma (\sigma_*\otimes \fs_\sigma)(\js_\tau)(\theta_{\sigma,\tau}\otimes 1 ))(J_{\sigma\tau}\otimes 1) \\
    = \Phi_\l(\Delta\otimes \id)(\js_\sigma (\sigma_*\otimes \fs_\sigma)(\js_\tau))(\Delta(\theta_{\sigma,\tau})J_{\sigma\tau}\otimes 1 )\\
   {\tiny\eqref{(vi)}} = \Phi_\l(\Delta\otimes \id)(\js_\sigma (\sigma_*\otimes \fs_\sigma)(\js_\tau))[(J_\sigma (\sigma_* \otimes \sigma_*)(J_\tau)(\theta_{(\sigma,\tau)}\otimes\theta_{(\sigma,\tau)} ))\otimes 1]\\
    = \Phi_\l(\Delta\otimes \id)(\js_\sigma (\sigma_*\otimes \fs_\sigma)(\js_\tau))[J_\sigma (\sigma_* \otimes \sigma_*)(J_\tau)\otimes 1]\\ \times [\theta_{(\sigma,\tau)}\otimes\theta_{(\sigma,\tau)} \otimes 1]
    \\
     {\tiny \eqref{eq 4}} = (\id\otimes \l)(\js_\sigma\overline{\circ}\ \js_\tau)(1\otimes \fs_\sigma\overline{\circ}\ \fs_\tau)((\sigma\tau)_*\otimes (\sigma\tau)_*\otimes \fs_\sigma\fs_\tau)(\Phi_\l)(\theta^{-1}\otimes\theta^{-1}\otimes 1)\\ \times [\theta_{(\sigma,\tau)}\otimes\theta_{(\sigma,\tau)} \otimes 1]\\
    = (\id\otimes \l)(\js_\sigma\overline{\circ}\ \js_\tau)(1\otimes \fs_\sigma\overline{\circ}\ \fs_\tau)((\sigma\tau)_*\otimes (\sigma\tau)_*\otimes \fs_\sigma\fs_\tau)(\Phi_\l)
\end{multline*}

The proof of the second equation is over.

(3)\ Now we shall prove the third equation:

\begin{multline*}
    \l(\fs_\sigma\circ \fs_\tau(x))\js_\sigma\overline{\circ}\ \js_\tau  \\ =\l(\fs_\sigma\circ \fs_\tau(x))\js_\sigma (\sigma_*\otimes \fs_\sigma)(\js_\tau)(\theta_{\sigma,\tau}\otimes 1 )\\
    {\tiny\eqref{twisted-end-comodalg3}}=\js_\sigma (\sigma_*\otimes \fs_\sigma)\l(\fs_\tau(x)) (\sigma_*\otimes \fs_\sigma)(\js_\tau)(\theta_{\sigma,\tau}\otimes 1 )\\
    =\js_\sigma (\sigma_*\otimes \fs_\sigma)[\l(\fs_\tau(x)) \js_\tau](\theta_{\sigma,\tau}\otimes 1 )\\
   {\tiny\eqref{twisted-end-comodalg3}} =\js_\sigma (\sigma_*\otimes \fs_\sigma)[\js_\tau(\tau_*\otimes\fs_\tau)\l(x) ](\theta_{\sigma,\tau}\otimes 1 )\\
    =\js_\sigma (\sigma_*\otimes \fs_\sigma)(\js_\tau) (\sigma_*\tau_*\otimes\fs_\sigma\fs_\tau)(\l(x))(\theta_{\sigma,\tau}\otimes 1 )\\
   {\tiny\eqref{(iii)}} =\js_\sigma (\sigma_*\otimes \fs_\sigma)(\js_\tau) (\theta_{\sigma,\tau}\otimes 1 )((\sigma\tau)_*\otimes\fs_\sigma\fs_\tau)(\l(x)) \\
    =(\js_\sigma\overline{\circ}\ \js_\tau)((\sigma\tau)_*\otimes\fs_\sigma\fs_\tau)\l(x) ]
\end{multline*}

Finally, we shall prove the associativity of $\overline{\circ} $,

\begin{multline*}
    [\js_\sigma \overline{\circ}\ \js_{\tau}]\overline{\circ}\ \js_{\rho}=
    [\js_\sigma (\sigma_*\otimes \fs_\sigma)(\js_\tau)(\theta_{\sigma,\tau}\otimes 1 )]\overline{\circ}\ \js_{\rho}\\
    =\js_\sigma (\sigma_*\otimes \fs_\sigma)(\js_\tau)(\theta_{\sigma,\tau}\otimes 1 )((\sigma\tau)_*\otimes (\fs_\sigma\circ\fs_\tau))(\js_\rho)(\theta_{\sigma\tau,\rho}\otimes 1)\\
    {\tiny\eqref{(iii)}}=\js_\sigma (\sigma_*\otimes \fs_\sigma)(\js_\tau)(\sigma_*\tau_*\otimes (\fs_\sigma\circ\fs_\tau))(\js_\rho) (\theta_{\sigma,\tau}\theta_{\sigma\tau,\rho}\otimes 1)\\
    {\tiny\eqref{(iv)}}=\js_\sigma (\sigma_*\otimes \fs_\sigma)(\js_\tau)(\sigma_*\tau_*\otimes (\fs_\sigma\circ\fs_\tau))(\js_\rho) (\sigma_*(\theta_{\tau,\rho})\theta_{\sigma,\tau\rho}\otimes 1)\\
    =\js_\sigma (\sigma_*\otimes \fs_\sigma)[\js_\tau (\tau_*\otimes \fs_\tau)(\js_\rho)(\theta_{\tau,\rho}\otimes 1)](\theta_{\sigma,\tau\rho}\otimes 1)\\
    =\js_\sigma (\sigma_*\otimes \fs_\sigma)[\js_{\tau}\overline{\circ}\ \js_{\rho}](\theta_{\sigma,\tau\rho}\otimes 1) = \js_\sigma \overline{\circ}\ [\js_{\tau}\overline{\circ}\ \js_{\rho}].
\end{multline*}\end{proof}

\subsection{Crossed system of comodule algebras}

Let $A$ be a quasi-Hopf algebra $(\kc, \l, \Phi_\l)$ be a left $A$-comodule algebra and $(\sigma_*, \theta_{(\sigma,\tau)}, J_\sigma)_{\sigma,\tau \in G}$ a $G$-crossed system over $A$.

We define the monoidal category $\underline{\Aut}^{\Tw}_G(\kc)$ of twisted automorphisms as follows. Objects in $\underline{\Aut}^{\Tw}_G(\kc)$ are  $(\sigma_*,J_\sigma)$-twisted automorphisms of $\kc$ for $\sigma \in G$, the set of arrows are the isomorphisms of twisted homomorphisms of $A$-comodule algebras, the tensor product of object is defined by the composition explained in Lemma  \ref{lema composicion}. The unity object is the $\id_\kc$, and tensor product of arrows is as in $\underline{\Aut}^{\Tw}(A)$.

\medbreak

Let $F\subset G$ be a subgroup. An $F$-\emph{crossed system for} a left $A$-comodule algebra $\kc$, compatible with the $G$-crossed system $(\sigma_*, \theta_{(\sigma,\tau)}, J_\sigma)_{\sigma,\tau \in G}$ is a monoidal functor $\overline{(\ )}: \underline{F} \to \underline{\Aut}^{\Tw}_G(\kc)$, that is, an $F$-crossed system consists of the following data:
\begin{itemize}
  \item A $(\sigma_*, J_\sigma)$-twisted automorphism $(\overline{\sigma},\overline{J_\sigma})$ for each $\sigma\in F$,
  \item an element $\overline{\theta_{(\sigma,\tau)}}\in \kc^{\times}$ for each $\sigma, \tau \in F$,
\end{itemize} such that
\begin{align}
\label{cross-com1}(\overline{1},\overline{J_1})&=(\id, 1\ot 1),\\
 \label{cross-com2}\overline{\theta}_{(\sigma,\tau)}\overline{(\sigma\tau)}(k)&= \overline{\sigma}(\overline{\tau}(k))\overline{\theta}_{(\sigma,\tau)},\\
  \label{cross-com3}\overline{\theta}_{(\sigma,\tau)}\overline{\theta}_{(\sigma\tau,\rho)} &= \overline{\sigma}(\overline{\theta}_{(\tau,\rho)}) \overline{\theta}_{\sigma,\tau\rho},\\
  \label{cross-com4}\overline{\theta}_{(1,\sigma)}&=\overline{\theta}_{(\sigma,1)}=1,\\
\label{cross-com5}\l(\overline{\theta}_{(\sigma,\tau)})\overline{J}_{\sigma\tau} &= \overline{J}_\sigma \big( (\sigma_*\otimes \overline{\sigma})(\overline{J}_\tau)\big) \theta_{(\sigma,\tau)}\otimes \overline{\theta}_{(\sigma,\tau)},
\end{align}
for all $k\in \kc$, $\sigma,\tau, \rho \in F$. Let $\kc\# F$ be the vector space $\kc\otk \ku F$ with product and coaction given by \begin{align}\label{prod-and-copr}(x\#\sigma)(y\# \tau)= x\overline{\sigma}(y)\overline{\theta}_{(\sigma,\tau)}\#\sigma\tau, \quad\delta(x\#\sigma)= x\_{-1}\overline{J}^1_\sigma\# \sigma\otimes x\_{0}\overline{J}^2_\sigma\# \sigma,\end{align}
for all $x, y\in\kc,$ $\sigma, \tau\in F$.
\begin{prop}
The foregoing operations make the space $\kc\# F$ into a left $A\#G$-comodule algebra with associator $\Phi_\delta=\Phi^1_\l\# 1\otimes \Phi^2_\l\# 1\otimes \Phi^3_\l\# 1$.\qed
\end{prop}

\begin{defi}\label{defi:crossed-p-comodalg}
Let $G$ be a group, and  $F\subseteq G$ be a subgroup. Let $A$ be a $G$-crossed product quasi-bialgebra, and  let $(\kc,\lambda, \Phi_\lambda)$ be a left $A$-comodule algebra. We shall say that $\kc$ is an $F$-\emph{crossed product}, if there is a decomposition $\kc= \bigoplus_{\sigma\in F} \kc_\sigma$, such that
\begin{itemize}
\item $\Phi_\lambda \in A_e\otimes A_e\otimes \kc_e$,
  \item $\kc_\sigma \kc_\tau \subseteq \kc_{\sigma\tau}$ for all $\sigma,\tau \in F$,
  \item $\kc_\sigma$ has an invertible element for each $\sigma \in F$,
  \item $\lambda(\kc_\sigma)\subseteq A_\sigma\otimes \kc_\sigma$ for each $\sigma \in F$.
\end{itemize}
\end{defi}

Let $A$ be a quasi-Hopf algebra and $(\sigma_*, \theta_{(\sigma,\tau)}, J_\sigma)_{\sigma,\tau \in G}$ be a crossed system for the group $G$. We have similar results as for quasi-Hopf algebras. The proof is analogous to the proof of Proposition \ref{graduado implica sistema cruzado}.

\begin{prop}\label{cross-comod-alg}
Let $(\mathcal{L},\delta)$ be a $F$-crossed $A\#G$-comodule algebra, for a subgroup $F\subseteq G$. Then there is an $A$-comodule algebra $\kc$, and an $F$-crossed system over $\kc$ compatible with the crossed system $(\sigma_*, \theta_{(\sigma,\tau)}, F_\sigma)_{\sigma,\tau \in G}$, such that $\kc\#F$ and $\mathcal{L}$ are isomorphic $A\#G$-comodule algebras.
\end{prop}
\begin{proof}
 Let $\mathcal{L}$ be a $F$-crossed $A\#G$-comodule algebra, for a subgroup $F\subseteq G$. Set $\kc=\mathcal{L}_e$. Since every $\mathcal{L}_\sigma$ has an invertible element, we may choose for each $\sigma \in F$ some invertible element $u_\sigma\in \mathcal{L}_\sigma$, with $u_e=1$. Then it is clear that $\mathcal{L}_\sigma =u_\sigma \mathcal{L}_e= \mathcal{L}_e u_\sigma$, and the set $\{u_\sigma: \sigma\in F\}$ is a basis for $\mathcal{L}$ as a left (and  right) $\mathcal{L}_e$-module.  Let us define the maps $$\overline{\sigma}(a)=u_\sigma au_\sigma^{-1}, \text{ for each $\sigma\ F$ and $a\in \mathcal{L}_e$,}$$ and $$\overline{\theta}:G\times G\to \mathcal{L}_e \;\; \text{ by } \; \overline{\theta}_{(\sigma,\tau)}=u_\sigma u_\tau u_{\sigma\tau}^{-1}\; \text{ for } \sigma, \tau \in F.$$

Note that $\{(1\#\sigma)\otimes u_\tau\}_{\sigma\in G,\tau\in F}$ is a basis for $A\#F\otimes\mathcal{L}$ as a left (and  right) $A\otimes\mathcal{L}_e$-module. We have that $\delta(u_\sigma)\in A\#\sigma\otimes \mathcal{L}_\sigma$ can be uniquely expressed as $\delta(u_\sigma)= \overline{J}_\sigma((1\#\sigma)\otimes u_\sigma)$, with $\overline{J}_\sigma \in A\otimes \mathcal{L}_e$, for all $\sigma\in F$.

Then, it is straightforward to see that the data $(\overline{\sigma},\overline{\theta}_{(\sigma,\tau)}, \overline{J}_\sigma)_{\sigma,\tau \in F}$, define an $F$-crossed system over the $A$-comodule algebra  $\mathcal{L}_e$, and $\mathcal{L}_e\# F$ is isomorphic to $\mathcal{L}$ as $A\#G$ comodule algebras.\end{proof}

Let $G$ be an Abelian group, $F\subseteq G$ a subgroup, $(\sigma_*,\theta_{(\sigma,\tau)}, J_\sigma)_{\sigma,\tau \in G}$ be a crossed system over a quasi-Hopf algebra $A$, and $(\overline{\sigma},\overline{\theta}_{(\sigma,\tau)}, \overline{J}_\sigma)_{\sigma,\tau \in F}$ be an $F$-crossed system for a $A$-comodule algebra $\kc$. We shall further assume that
\begin{equation}\label{sym-theta} \theta_{(\sigma,\tau)}=\theta_{(\tau,\sigma)}, \quad \overline{\theta}_{(\rho,\nu)}=\overline{\theta}_{(\nu,\rho)},  \end{equation}
for all $\sigma, \tau\in G$, $\rho,\nu\in F$. We can consider the action of $G$ on the category  $\Rep(A)$ described in Lemma \ref{g-crossed-action}.
\begin{prop}\label{equiv-mod-c} Under the above assumptions the following assertions hold.
\begin{itemize}
  \item[1.] The $\Rep(A)$-module category $ {}_\kc\Mo$ is $F$-equivariant.
  \item[2.] There is an equivalence  between $(_\kc\Mo)^F$ and $_{\kc\#F}\Mo$ as $\Rep(A)^G$-module categories.
\end{itemize}
\end{prop}
\pf 1. For any $\rho\in F$ define $(U_\rho, c^\rho): {}_\kc\Mo\to ( {}_\kc\Mo)^\rho$ the $\Rep(A)$-module functor given as follows. For any $M\in {}_\kc\Mo$, $U_\rho(M)=M$ as vector spaces and the action of $\kc$ is given by: $x\cdot v= \overline{\rho}(x)\cdot v$, for all $x\in\kc, v\in M$. For any $X\in \Rep(A)$, $M\in {}_\kc\Mo$ the maps $c^\rho_{X,M}:U_\rho(X\otk M)\to F_\rho(X)\otk  U_\rho(M)$ are defined by $c^\rho_{X,M}(x\ot v)=
\overline{J}^{-1}_\rho\cdot (x\ot v)$, for any $x\in X, v\in M$. Equation \eqref{modfunctor1} for the pair $(U_\rho,c^\rho)$ follows from \eqref{twisted-end-comodalg2}.

For any $\sigma, \tau\in F$ define $\mu_{\sigma,\tau}:U_\sigma\circ U_\tau, \to U_{\sigma\tau}$ as follows. For any $M\in   {}_\kc\Mo$, $m\in M$
$$\mu_{\sigma,\tau}(m)=\overline{\theta}^{-1}_{(\sigma,\tau)}\cdot m. $$
It follows from equation \eqref{(iii)} that $\mu_{\sigma,\tau}$ is a morphism of $\kc$-modules. Equation \eqref{mod-equi1} follows from \eqref{(iii)} and  \eqref{mod-equi2} follows from \eqref{(vi)}.

\medbreak

2. Let $\mathcal{T}:(_\kc\Mo)^F\to {}_{\kc\#F}\Mo$ be the module functor defined as follows. If $(M,v)$ is an $F$-equivariant object then for any $\sigma\in F$ we have isomorphisms $v_\sigma:U_\sigma(M)\to M$ satisfying
$$v_{\sigma\tau}(\theta^{-1}_{(\sigma,\tau)}\cdot m)=v_\sigma(v_\tau(m), \quad v_\sigma(\overline{\sigma}(x)\cdot m)=x \cdot v_\sigma(m),$$
for all $\sigma, \tau\in F$, $x\in \kc$, $m\in M$. In this case there is a well-defined action of $\kc\#F$ on $M$ determined by
\begin{align}\label{action-equi}(x\# \sigma)\cdot m= x\cdot v^{-1}_\sigma(m),
\end{align}
for all $\sigma\in F$, $x\in \kc$, $m\in M$. We define $\mathcal{T}(M)=M$ with the above described action. If $(X,u)\in \Rep(A)^G$, $(M,v)\in (_\kc\Mo)^F$ the action of $\kc\#F$ on $X\ot M$ using the coaction given in \eqref{prod-and-copr} coincides with the action \eqref{action-equi} using the isomorphism $\widetilde{v}$ described in Lemma \ref{equivariant-mod-cat}. The proof that $\mathcal{T}$ is an equivalence is analogous to the proof of Proposition \ref{equi}.
\epf

The category of $F$-equivariant objects in a module category is always of the form ${}_{\kc\# F}\Mo$ for some left $A$-comodule algebra $\kc$ and some group $F$.

\begin{prop}\label{main-s5}
Let $A$ be a finite dimensional quasi-Hopf algebra and $G$ be a finite Abelian group and  $F\subset G$ a subgroup. Let $(\sigma_*,\theta_{(\sigma,\tau)}, J_\sigma)_{\sigma,\tau \in G}$ be a $G$-crossed system over $A$, and $\Mo$ be an exact $F$-equivariant $\Rep(A)$-module category. Then there is a left $A$-comodule algebra $(\kc,\lambda, \Phi^\lambda)$  such that ${}_{\kc}\Mo\cong \Mo$ as $\Rep(A)$-module categories and there is an $F$-crossed system compatible with $(\sigma_*,\theta_{(\sigma,\tau)}, J_\sigma)_{\sigma,\tau \in G}$ such that ${}_{\kc\# F}\Mo\simeq \Mo^F$ as $\Rep(A)^G$-module categories.
\end{prop}
\begin{proof} Let $B$ be a finite-dimensional quasi-Hopf algebra such that there is a quasi-Hopf algebra projection $\pi:B\to A$ and an equivalence $\Rep(B)\simeq \Rep(A)\rtimes F$ of tensor categories, see section \ref{section:Ggraded}. Since $\Mo$ is $F$-equivariant follows from Proposition \ref{mod-cat-equi} that $\Mo$ is an exact $\Rep(B)$ module category.

Hence there exists a left $B$-comodule algebra $(\kc,\lambda, \Phi^\lambda)$ such that $\Mo\simeq {}_\kc\Mo$ as $\Rep(B)$-modules. Let us recall that the equivariant structure is given by
$$(U_\sigma, c^\sigma):\Mo\to \Mo^\sigma,\quad U_\sigma(M)=[\uno,\sigma]\otb M,$$
for all $\sigma\in F$, $M\in \Mo$ together with a family of natural isomorphisms $\mu_{\sigma,\tau}:U_\sigma\circ U_\tau\to U_{\sigma\tau}$ for any $\sigma, \tau\in F$.  Under the equivalence $\Rep(B)\simeq \Rep(A)\rtimes F$ the object $[\uno,\sigma]$ correspond to a 1-dimensional representation of $B$. For any $\sigma\in F$ let us denote by $\chi_\sigma:B\to \ku$ the corresponding character and the algebra map $\overline{\sigma}:\kc\to \kc$, $\overline{\sigma}(k)=\chi_\sigma(k\_{-1})\, k\_0$, for all $k\in\kc$.
\medbreak

Define $\lambda^\pi=(\pi\ot\id) \lambda$, then $(\kc,\lambda^\pi, (\pi\ot \pi\ot\id)(\Phi^\lambda))$ is a left $A$-comodule algebra that we will denote by $\kc^\pi$. The equivalence $\Mo\simeq {}_\kc\Mo$ of $\Rep(B)$-module categories induces an equivalence $\Mo\simeq {}_{\kc^\pi}\Mo$ of $\Rep(A)$-modules. Under this equivalence the functors $U_\sigma:{}_{\kc^\pi}\Mo\to ({}_{\kc^\pi}\Mo)^\sigma$ are given as follows. For any $M\in {}_{\kc^\pi}\Mo$, $U_\sigma(M)=M$ and the action of $\kc$ on $M$ is given by
$$k\cdot m = \overline{\sigma}(k)\cdot m,\quad \text{ for all } k\in\kc, m\in M.$$
For any $\sigma, \tau\in F$ denote
$$\overline{J}_\sigma=c_{A,\kc}^\sigma(1\ot 1)^{-1}, \quad \overline{\theta}_{\sigma,\tau}=(\mu_{\tau,\sigma})_{\kc}(1)^{-1}.$$
Turns out that the collection $(\overline{\sigma},\overline{\theta}_{(\sigma,\tau)}, \overline{J}_\sigma)_{\sigma,\tau \in F}$ is an $F$-crossed system compatible with $(\sigma_*,\theta_{(\sigma,\tau)}, J_\sigma)_{\sigma,\tau \in G}$ for the $A$-comodule algebra $\kc^\pi$. Indeed  for any $\sigma\in F$ the pair $(\overline{\sigma},\overline{J_\sigma})$ is a $(\sigma_*, J_\sigma)$-twisted automorphism since equation \eqref{twisted-end-comodalg2} follows from the fact that $c^\sigma$ satisfies \eqref{modfunctor1} and equation \eqref{twisted-end-comodalg3} follows since $c^\sigma$ is a $\kc$-module morphism. Equation \eqref{cross-com2} follows since $\mu_{\sigma,\tau}$ is a morphism of $\kc$-modules, equation \eqref{cross-com3} follows from \eqref{mod-equi1} and
equation \eqref{cross-com5} follows from \eqref{mod-equi2}. The equivalence ${}_{\kc^\pi\# F}\Mo\simeq \Mo^F$ as $\Rep(A)^G$-module categories follows from Proposition \ref{equiv-mod-c}.

\end{proof}

\section{Module categories over the quasi-Hopf algebras $A(H,s)$}\label{quasi-basi-ex}

\subsection{Basic Quasi-Hopf algebras $A(H,s)$}\label{s:basicquasi}

We recall the definition of a the family of basic quasi-Hopf algebras $A(H,s)$ introduced
 by  I. Angiono \cite{A} and used to give a classification of pointed tensor categories with cyclic group of
invertible objects of order $m$ such that $210\nmid m$.

\medbreak

Let $m\in \Na$ and $H=\oplus_{n \geq 0} H(n)$ be a finite-dimensional radically graded pointed Hopf algebra  generated by a group like element $\chi$ of order $m^{2}$ and skew primitive elements $x_1,...,x_{\theta}$ satisfying
\begin{equation}\label{skewprimitives}
\chi x_i\chi^{-1} = q^{d_i}x_i, \quad \Delta(x_i)= x_i \otimes 1 + \chi^{-b_i}\otimes x_i,
\end{equation}
for any $i=1,\dots ,\theta$, where $q$ is a primitive root of 1 of order $m^{2}$,
$H=\nic(V) \# \ku C_{m^{2}}$, where  $\nic(V) $ is the associated Nichols algebra of the Yetter-Drinfeld module $V \in \gyd$.

\medbreak

We shall further assume that $\nic(V) $ has a basis $\{x_1^{s_1}\dots x_\theta^{s_\theta}: 0\leq s_i\leq N_i\}.$

\begin{rmk} The above condition does not hold for any Nichols algebra. If $V$ has diagonal braiding with
Cartan matrix of type $A_3$ then $\nic(V) $ is not generated by elements of degree 1. This conditions is satisfied for
example for any quantum linear space.
\end{rmk}

Set $\sigma:=\chi^{m}$ and denote by $\{1_i:i\in C_{m^{2}}\},$ $\{\uno_j: j\in C_m\}$ the families of primitive idempotents in $\ku  C_{m^{2}}$ and $\ku C_m$ respectively. That is
$$1_i=\frac{1}{m^{2}}\sum_{k=0}^{m^2-1}\; q^{-ki}\, \chi^k, \quad \uno_j=\frac{1}{m}\sum_{l=0}^{m-1}\; q^{-mlj}\, \sigma^l.$$

For any $0\leq s\leq m-1$ set $J_s= \sum_{i,j=0}^{m^{2}-1} c(i,j)^s 1_i\otimes1_j$, where $c(i,j):= q^{j(i-i')}$. Here $j'$ denotes the remainder in the division by $m$. The associator $\Phi_s=dJ_s$ is written explicitly  as
\begin{equation}\label{dassociator}
    \Phi_s:= \sum_{i,j,k=0}^{m-1} \omega_s(i,j,k) \uno_i \otimes \uno_j \otimes \uno_k,
\end{equation}
where $\omega_s: (C_{m})^3 \rightarrow \ku^{\times}$ is  the 3-cocycle defined by $\omega_s(i,j,k) = q^{sk(j+i-(j+i)')}$. Consider the quasi-Hopf algebra $(H_{J_s}, \Phi_s)$ obtained by twisting $H$. Denote $\Upsilon(H)= \{1\leq s \leq m-1: b_i \equiv s d_i \; \text{mod}(m), \, 1 \leq i \leq \theta \}$. For any $s\in \Upsilon(H)$ the quasi-Hopf algebra $A(H,s)$ is defined as the subalgebra of $H$ generated by $\sigma$ and $x_1,...,x_\theta$. The algebra $A(H,s)$ is a quasi-Hopf subalgebra of $H_{J_s}$ with associator $\Phi_s$ such that $A(H, s)/ \Rad A(H,s) \cong \ku[C_{m}]$. See \cite[Prop. 3.1.1]{A}.

For any $1\leq i \leq \theta$ we have that
\begin{align*}\Delta_{J_s}(x_i)&= \sum_{y=0}^{m-1} q^{b_iy}  \uno_{y} \ot x_i + \sum_{z=0}^{m-1} \left(\sum_{y=0}^{m-d_i'-1} q^{(d_i'-d_i)sz}x_i \uno_{y} \ot \uno_{z} \right.
  \\ &\left. + \sum_{j=m-d_i'}^{m-1} q^{(d_i'+m-d_i)sz} x_i \uno_{y} \ot \uno_{z} \right).
\end{align*}

\begin{rmk} Our definition of $A(H,s)$ is slightly different that the one given in \cite[$\S$ 3]{A}. This is not a problem since our quasi-Hopf algebras are isomorphic to the ones defined in \textit{loc. cit.} except that the $s$ may change. The difference comes from the fact that we are using $(\chi^{-b_i}, 1)$ skew-primitive elements instead of $(1,\chi^{b_i})$ skew-primitive elements.
\end{rmk}

\subsection{$C_m$-crossed system over $A(H,s)$} The cyclic group with $m$ elements will be denoted by $C_m=\{1,h,h^2,\dots, h^{m-1}\}$. For any $0\leq i < m$ set $((h^i)_*,J_{h^i})$ the twisted endomorphism of   $A(H,s)$ given by $$J_{h^i}=1\ot 1,\quad (h^i)_*(a)=\chi^{i'} a\chi^{-i'}\quad \text{ for all } a\in A(H,s).$$
For any $0\leq i, j < m$ define $\theta_{(i,j)}=\theta_{(h^i,h^j)}= \sigma^{\frac{(i+j)-(i+j)'}{m}}$.

\begin{rmk} If $i+j\leq m$ then $\theta_{(i,j)}=1$ and if $i+j>m$ then $\theta_{(i,j)}=\sigma$. In principle the algebra maps $(h^i)_*$ are defined in $H$ but when restricted to $A(H,s)$ they are well-defined.\end{rmk}

These data is a $C_m$-crossed system over $A(H,s)$ such that the equivariantization $\Rep(A)^{C_m}$ is tensor equivalent to $\Rep(H)$. This is contained in the next result which gives an alternative proof for \cite[Thm. 4.2.1]{A}.

\begin{prop}\label{crossed-a} \begin{itemize}
               \item[1.] $((h^i)_*,\theta_{(i,j)},J_{h^i})_{h^i,h^j\in C_m}$ is a $C_m$-crossed system over $A(H,s)$.
               \item[2.] There is an isomorphism of quasi-Hopf algebras $A(H,s)\# C_m\simeq H_{J_s}.$
               \item[3.]   There is a tensor equivalence $\Rep(A)^{C_m}\simeq \Rep(H)$.

             \end{itemize}
\end{prop}
\pf 1. it follows by a straightforward computation.

\smallbreak

2. Define $\varphi: A(H,s)\# C_m \to  H_{J_s}$ the linear map given by
$$\varphi(a\# h^i)=a \chi^{i'}, $$
for all $0\leq i < m$, $a\in A$. Let $0\leq i, j < m$, $a, b\in A$ then
\begin{align*}\varphi((a\# h^i)(b\# h^j))&=\varphi(a (h^i)_*(b)\theta_{(i,j)}\#h^{i+j} )=a \chi^{i'} b\chi^{-i'} \theta_{(i,j)}\chi^{(i+j)'}.
\end{align*}
On the other hand
$$ \varphi(a\# h^i) \varphi(b\# h^j)= a \chi^{i'} b\chi^{j'}.$$
It is enough to prove that $\chi^{i'} b\chi^{j'}= \chi^{i'} b\chi^{-i'} \theta_{(i,j)}\chi^{(i+j)'}$ for $b=x_l$, $1\leq l\leq \theta$. If $i+j\leq m$ then
$$\chi^{i'} x_l\chi^{-i'} \theta_{(i,j)}\chi^{(i+j)'}=q^{d_l i}\, x_l\chi^{i+j}=\chi^i x_l\chi^j.$$
If $i+j=m+k$, $k>0$ then
$$\chi^{i'} x_l\chi^{-i'} \theta_{(i,j)}\chi^{(i+j)'}=q^{d_l i}\, x_l \sigma \chi^k=q^{d_l i}\, x_l  \chi^{i+j}= \chi^i x_l\chi^j.$$
It follows immediately that $\varphi$ is a coalgebra map and it is injective and by a dimension argument is bijective.

\smallbreak

3. It follows from Proposition \ref{equi}.\epf

\begin{rmk} There is a grading on $H$ compatible with the isomorphism of Proposition \ref{crossed-a} (2). Namely, if $\sigma\in G$ then the vector space $H_\sigma$ has basis $\{x_1^{s_1}\dots x_\theta^{s_\theta} \sigma\}$. Define $H^{(i)}=\oplus_{j=0}^{m-1}\, H_{\chi^{mj+i}}$, thus $H=\oplus_{j=0}^{m-1} \,H^{(j)}$. It is not difficult to prove that with this grading $H$ is a $C_m$-crossed product (see definition \ref{defi:crossedp}) and this crossed product is compatible with the isomorphism of Proposition \ref{crossed-a} (2).
\end{rmk}

\subsection{Right simple $A(H,s)$-comodule algebras} We shall present some families of right $A(H,s)$-simple left $A(H,s)$-comodule algebras. This class will be big enough to classify module categories over $\Rep(A(H,s))$ in some cases.

\medbreak

Let $(K,\lambda)$ be a finite-dimensional left $H$-comodule algebra. We say that  $(K,\lambda)$ is \emph{of type 1} if the following assumptions are satisfied:
\begin{itemize}

  \item There exists a subgroup $F\subseteq C_{m^2}$ and $t\in\Na$ such that $K$ has a basis $\{y_1^{r_1}\dots y_t^{r_t} e_f: 0\leq r_j< N_j,  f\in F, t\leq\theta\}$ such that
  $$\quad e_{\chi^a}y_l= q^{ad_l}  y_l e_{\chi^a},\quad \text{ if }\; \chi^a\in F.$$

  \item there is an inclusion $\iota:K\hookrightarrow H$ of $H$-comodules such that
      $$\iota(e_f)=f, \quad \iota(y_l)=x_l,$$
      for all $f\in F$, $l=1\dots t$.
\end{itemize}
Observe that in this case we have that
$$ \lambda(e_f)=f\ot e_f, \quad \lambda(y_l)=x_l\ot 1+ \chi^{-b_l}\otimes y_l.$$

\begin{defi}\label{type1-hopf} We shall say that a Hopf algebra $H=\nic(V)\#\ku G$ is of type 1 if \begin{enumerate}
            \item $\nic(V) $ has a basis $\{x_1^{s_1}\dots x_\theta^{s_\theta}: 0\leq s_i\leq N_i\},$ where $V$ is the vector space generated by $\{x_1,\dots,x_\theta\}$,
            \item any right $H$-simple left $H$-comodule algebra $(K,\lambda)$ is equivariantly Morita  equivalent to a comodule algebra of type 1.
          \end{enumerate}
\end{defi}

\begin{rmk} If $H=\nic(V)\# \ku\Gamma$ is the bosonization of a Nichols algebra and a group algebra
 a finite group $\Gamma$ then $H$ is of type 1 when $V$ is a quantum linear space and $\Gamma$ is
 an Abelian group \cite{Mo2} or when $V$ is constructed from a rack and $\Gamma=\sy_3,\sy_4$ \cite{GM}.
\end{rmk}

Let $(K, \lambda)$ be a type 1 left $H$-comodule algebra such that $K_0=\ku F$ where $F\subseteq C_{m^2}$ is a subgroup such that $  <\sigma>\subseteq F$ we shall denote by $\lambda^{J_s}:K\to  H\ot K$ the map given by
$$ \lambda^{J_s}(x)=J_s\lambda(x) J^{-1}_s,\quad \text{ for all }\, x\in K.$$
Here $J_s$ is identified with an element in $H\ot K$ via the inclusion $\id_H\ot \iota$. The same calculation as in \cite[Prop. 3.1.1]{A} proves that $\lambda^{J_s}(K)\subseteq H\ot K$. Define $(K^{J_s},\lambda^{J_s},\Phi_s(J_s\ot 1))$  the left $H$-comodule algebra with underlying algebra  $K^{J_s}$, coaction $\lambda^{J_s}$ and associator $\Phi_s(J_s\ot 1)$.
It follows from Lema \ref{comod-alg:twisting} that  $(K^{J_s},\lambda^{J_s},\Phi_s)$ is a left $H_{J_s}$-comodule algebra.

\begin{lema}\label{invariance-dynamical-twist} The left $H$-comodule algebras $(K, \lambda)$ and $(K^{J_s},\lambda^{J_s},\Phi_s(J_s\ot 1))$ are equivariantly Morita equivalent, that is  ${}_K\Mo$, ${}_{K^{J_s}}\Mo$ are equivalent as $\Rep(H)$-modules.\qed
\end{lema}
\pf For any $X\in \Rep(H)$, $M\in {}_K\Mo$ and any $x\in X$, $m\in M$ define
$$c_{X,M}:X\otk M\to X\otk M, \quad c_{X,M}(x\ot m)=J_s\cdot (x\ot m).$$
It is immediate to prove that the identity functor $(\Id, c):{}_K\Mo\to {}_{K^{J_s}}\Mo$ is an equivalence of module categories.\epf

\begin{defi} Let $(\kc, \lambda , \Phi^\lambda)$ be a left $H_{J_s}$-comodule algebra such that the associator $\Phi^\lambda\in A(H,s)\otk A(H,s)\otk \kc$. Define $\widehat{\kc}=\lambda^{-1}(A(H,s)\otk \kc)$ and denote $\widehat{\lambda}$ the restriction of $\lambda$ to $\widehat{\kc}$. Then $(\widehat{\kc},\widehat{\lambda},\Phi_s)$ is a left $A(H,s)$-comodule algebra. Turns out that this procedure is the inverse of the crossed product.\end{defi}

\subsection{Actions on module categories ${}_{(\widehat{K},\widehat{\lambda},\Phi_s)}\Mo$} For the rest of this section we shall assume now that $m=p$ is a prime number.

Let $(K, \lambda)$ be a type 1 left $H$-comodule algebra such that $K_0=\ku F$ where $F=C_d$ is a cyclic group.

\medbreak

There are two possible cases; when  $<\sigma>\subseteq F$ or $F=\{1\}$. Let us treat the first case. So we assume that  $p|d$. Let $s,l\in\Na$ be such that $d=ps$ and $sl=p$. Let us denote $ \widehat{F}=C_s=<\chi^{lp}>$.

\medbreak

By hypothesis the vector space $K$ has a decomposition $K=\oplus_{f\in F}\, K_f$ where $K_f$ is the vector space with basis $\{y_1^{r_1}\dots y_t^{r_t} e_f: 0\leq r_j\leq N_j\}$. For any $i=0\dots s-1$ define
$$K^{(i)}=\bigoplus_{j: \,\chi^{mj+i}\in C_d}\,\, K_{\chi^{mj+i}}.$$
Observe that $\widehat{K}=K^{(0)}$. With this grading $K$ is an $\widehat{F}$-crossed product.

\begin{lema}\label{action-modcat}  Under the above assumptions ${}_{(\widehat{K},\widehat{\lambda},\Phi_s)}\Mo$ is an $\widehat{F}$-equivariant  $\Rep(A(H,s))$-module category  and $\big({}_{(\widehat{K},\widehat{\lambda},\Phi_s)}\Mo\big)^{\widehat{F}} \simeq {}_K\Mo$ as module categories over $\Rep(H)$.
\end{lema}
\pf It follows from Proposition \ref{cross-comod-alg} and Proposition \ref{equiv-mod-c}.\epf

Now, let us assume that $F=\{1\}$. Let us endow the space $K\otk \ku C_p$ with the product determined by
$$(y_l\ot \sigma^a)(y_s\ot \sigma^b)= q^{p a d_s}\; y_l y_s\ot \sigma^{a+b}. $$
The space $K\otk \ku C_p$  is a left $H$-comodule algebra with coproduct determined by
$$\lambda(y_l\ot \sigma^a)= x_l \sigma^a\ot 1\ot \sigma^a + \sigma^a\chi^{-b_l}\ot y_l\ot \sigma^a.$$
It is clear that $(K\otk \ku C_p)_0=\ku C_p.$ Thus we can consider the left $A(H,s)$-comodule algebra $(K\otk \ku C_p,\widehat{\lambda},\Phi_s).$

\begin{lema}\label{another-action} Under the above conventions the following holds. \begin{itemize}
                   \item[1.]The module category ${}_{( \ku C_p,\lambda,\Phi_s)}\Mo$ has a $C_p$-action  such that there is an equivalence $\big({}_{( \ku C_p,\lambda,\Phi_s)}\Mo\big)^{C_p}\simeq \Vect_\ku$ as $\Rep(H)$-modules.
                   \item[2.] The module category ${}_{(K\otk \ku C_p,\widehat{\lambda},\Phi_s)}\Mo$ has a $C_p$-action  such that there is an equivalence $\big({}_{(K\otk \ku C_p,\widehat{\lambda},\Phi_s)}\Mo\big)^{C_p}\simeq {}_K\Mo$ as $\Rep(H)$-modules.
                 \end{itemize}

\end{lema}
\pf 1. It follows from (2) taking $K=\ku$.
\medbreak

2. Set $\Mo={}_{(K\otk \ku C_p,\widehat{\lambda},\Phi_s)}\Mo$. For any $i=0,\dots,p-1$ define the functors $(U_i, c^i):\Mo\to \Mo^{\sigma^i}$ as follows. For any $M\in \Mo$ $U_i(M)=M$ with a new action $\rhd:(K\otk \ku C_p)\otk M\to M$ of $K\otk \ku C_p$ given by
$$y_l\rhd m= q^{i d_l}\, y_l\cdot m, \quad \sigma\rhd m = q^{i p}\; \sigma\cdot m,$$
for all $l=1,\dots, t$, $m\in M$. For any $X\in \Rep(A), M\in \Mo$ the map $c^i_{X,M}:U_i(X\otk M)\to  F_i(X)\otk U_i(M)$ is the identity.

The isomorphism $\mu_{i,j}:U_i\circ U_j\to U_{i+j}$ is given by the action of $\sigma^{-\frac{(i+j)-(i+j)'}{p}}$. Altogether makes the category ${}_{(K\otk \ku C_p,\widehat{\lambda},\Phi_s)}\Mo$ a $C_p$-equivariant $\Rep(A)$-module category.
\medbreak

Let $N\in {}_K\Mod$. Define $\Fc(N)=\oplus_{i=0}^{p-1}\, N_i$ where $N_i=N$ as vector spaces. Let us define a new action of $\rightharpoonup:K\otk \ku C_p\otk \Fc(N)\to \Fc(N)$ as follows. If $n\in N_i$ then
$$\sigma\rightharpoonup n= q^{pi}\; n\in N_i,\quad y_l\rightharpoonup n=  q^{d_li}\; y_l\cdot n\in N_{(d_l+i)'}.$$
Recall that $a'$ denotes the remainder of $a$ in the division by $p$. Note also that for any $i,j=0,\dots,p-1$ $U_i(N_j)=N_{i+j}$. The module $\Fc(N)$ is a $C_p$-equivariant object in ${}_{(K\otk \ku C_p,\widehat{\lambda},\Phi_s)}\Mo$, indeed for any $i=0,\dots,p-1$ define the isomorphisms $v_i:U_i(\Fc(N))\to \Fc(N)$ as follows: $v_i(n)=q^{-i}\; n \in N_{i+j}$ for any $n\in N_j$  . This maps are $K\otk \ku C_p$-module isomorphisms and they satisfy equation \eqref{F-equivariant-obj}. This defines a functor $\Fc:{}_K\Mod\to{}_{(K\otk \ku C_p,\widehat{\lambda},\Phi_s)}\Mo$ that together with the identity isomorphisms $c_{X, N}: \Fc(X\otk N)\to X\otk \Fc(N)$ becomes a module functor.

If $M\in {}_{(K\otk \ku C_p,\widehat{\lambda},\Phi_s)}\Mo$ then $M=\oplus_{i=0}^{p-1}\, M_i$ where $M_i$ is the eigenspace of the eigenvalue $q^{pi}$ of the action of $\sigma$. The space $M_0$ has a $K$-action as follows. Since $M$ is $C_p$-equivariant there are isomorphisms $v_i:U_i(M)\to M$ such that the restrictions $v_i\mid_{M_0}:M_0\to M_i$ are isomorphisms. If $m\in M_0$, $y_l\in K$ then $y_l\cdot m\in M_{d_l}$, thus we can define $\rightharpoonup:K\otk M_0\to M_0$
$$ y_l\rightharpoonup m= v^{-1}_{d_l}(y_l\cdot m),$$
for all $m\in M_0$. The map $M\mapsto M_0$ is functorial and defines an inverse functor for $\Fc$.
\epf

\subsection{Exact module categories over $\Rep(A(H, s))$}

Now we can formulate the main result of this section.

\begin{teo}\label{main6} Let $H$ be a Hopf algebra of type 1 (see definition \ref{type1-hopf}) and let $\Mo$ be an exact indecomposable module category over $\Rep(A(H, s))$. Then the following statements hold.

\begin{enumerate}
  \item there exists a right $H$-simple left $H$-comodule algebra $(K, \lambda)$ with trivial coinvariants such that $K_0\supseteq C_p$ and there is an equivalence of module categories $\Mo\simeq {}_{(\widehat{K},\widehat{\lambda},\Phi_s)}\Mo$.
  \item If there is an equivalence $ {}_{(\widehat{K'},\widehat{\lambda'},\Phi'_s)}\Mo\simeq {}_{(\widehat{K},\widehat{\lambda},\Phi_s)}\Mo$ as $\Rep(A(H,s))$-modules then $(K, \lambda)$ and $(K', \lambda')$ are equivariantly Morita equivalent $H$-comodule algebras.
\end{enumerate}

\end{teo}
\pf

1. By Lemma \ref{exact-mod-quasi} there exists a left $A(H, s)$-comodule algebra $(\kc, \lambda,\Phi)$ such that $\Mo\simeq {}_\kc\mo$. The category ${}_\kc\mo$ is $F$-equivariant for some subgroup $F\subseteq C_{p}$. Thus it follows from \cite[Thm 3.3]{AM} that there is a right $H$-simple left $H$-comodule algebra $(S, \delta)$ with trivial coinvariants such that $\big({}_\kc\mo\big)^F\simeq {}_S\Mo$ as $\Rep(H)$-modules. Hence $S_0=\ku 1$, $S_0=\ku C_p$ or $S_0=\ku C_{p^2}.$ In any case, it follows from Lemmas \ref{action-modcat}, \ref{another-action} that there is a right $H$-simple left $H$-comodule algebra $(K, \lambda)$ with trivial coinvariants such that $K_0\supseteq C_p$ and there is an equivalence  ${}_S\Mo\simeq\big({}_{(\widehat{K'},\widehat{\lambda'},\Phi'_s)}\Mo\big)^F$. Whence $\big({}_\kc\mo\big)^F \simeq\big({}_{(\widehat{K'},\widehat{\lambda'},\Phi'_s)}\Mo\big)^F$, thus using Proposition \ref{mod-cat-equi} (5) we get the result.
\medbreak

2. There exists a subgroup $F\subseteq C_p$ such that both module categories $ {}_{(\widehat{K'},\widehat{\lambda'},\Phi'_s)}\Mo, {}_{(\widehat{K},\widehat{\lambda},\Phi_s)}\Mo$ are $F$-equivariant and there are equivalences of module categories over $\Rep(H_{J_s})$
$${}_{(K,\lambda,\Phi_s)}\Mo\simeq\big({}_{(\widehat{K},\widehat{\lambda},\Phi_s)}\Mo\big)^F
\simeq\big({}_{(\widehat{K'},\widehat{\lambda'},\Phi'_s)}\Mo \big)^F
\simeq{}_{(K',\lambda',\Phi'_s)}\Mo.$$ Thus by Lemma \ref{invariance-dynamical-twist} follows that ${}_K\Mo\simeq {}_{K'}\Mo$.
\epf

\subsection{Some classification results}
We apply Theorem \ref{main6} to obtain the classification of module categories over $\Rep(A(H, s))$ where $H$ is the bosonization of a quantum linear space.

\medbreak

Let  $g_1,
\dots, g_{\theta}\in C_{p^2}$, $\chi_1,
\dots,\chi_{\theta}\in\widehat{C_{p^2}}$ be a datum for a quantum linear space and let $V=V(g_1,\dots,
g_{\theta},\chi_1, \dots,\chi_{\theta})$ the associated Yetter-Drinfeld
module over $\ku C_{p^2}$ generated as a vector space by
$x_1,\dots, x_{\theta}$. For more details see \cite{AS}.

The Hopf algebra $H=\nic(V)\# \ku C_{p^2}$ is a type 1 Hopf algebra, see \cite{Mo2}.

\medbreak

Let us define now a family of right $H$-simple left $H$-comodule
algebras. Let $F\subseteq C_{p^2}$ be a subgroup and $\xi=(\xi_{i})_{i=1\dots \theta}$,
$\alpha=(\alpha_{ij})_{1\leq i<j\leq\theta}$ be two families of
elements in $\ku$ satisfying
\begin{align}\label{parameters1} \xi_{i}=0\; \text{ if }\;
g^{N_i}_i\notin F  \text{ or }\;\chi^{N_i}_i(f)\neq
1,
\end{align}
\begin{align}\label{parameters2}
 \alpha_{ij}=0\; \text{ if }\;
g_ig_j \notin F  \text{ or }\;
 \chi_i\chi_j(f)\neq 1,
\end{align}
for all $f\in F$. In this case we shall say that the pair $(\xi,
\alpha)$ is a \emph{compatible comodule algebra datum} with respect
to the quantum linear space $V$ and the
group $F$.

 The algebra $\Ac(V, F, \xi, \alpha)$ is the algebra
generated by elements in $\{v_i:i=1\dots \theta\}$, $\{e_f: f\in
F\}$ subject to relations
\begin{align}\label{relations1} e_f e_g =  e_{fg},\quad
e_f v_i =  \chi_i(f)\; v_i e_f,
\end{align}
\begin{align}\label{relations2}  v_i v_j - q_{ij}\; v_jv_i=\begin{cases}
  \alpha_{ij}\; e_{ g_ig_j}\;\; \;\text{ if }
  g_ig_j\in F\\
0 \;\;\; \;\;\;\quad\quad\text{ otherwise, }
\end{cases}
\end{align}
\begin{align}\label{relations3} v_i^{N_i}=\begin{cases}  \xi_{i}\; e_{ g_i^{N_i}}\;\; \;\text{ if } g_i^{N_i}\in F\\
0 \;\;\; \;\;\;\quad\quad\text{ otherwise, }
\end{cases}
\end{align}
for any $1\leq i< j\leq \theta$.
If $W\subseteq V$ is a $\ku C_{p^2}$-subcomodule invariant under the
action of $F$, we define $\Ac(W, F, \xi, \alpha)$ as the
subalgebra of $\Ac(V, F, \xi, \alpha)$ generated by $W$ and
$\{e_f: f\in F\}$.

The algebras $\Ac(V, F, \xi, \alpha)$ are right $H$-simple left $H$-comodule algebras with coaction determined by
$$\lambda(v_i)=x_i\ot 1+ g_i\ot v_i,\quad \lambda(e_f)=f\ot e_f,$$
for all $i=1,\dots,\theta$, $f\in F$. The subalgebras $\Ac(W, F, \xi, \alpha)$ are also right $H$-simple left $H$-subcomodule algebras.

\begin{teo}\label{clasification-qspaces}\cite[Thm 4.6, Thm. 4.9]{Mo2}
Let $\Mo$ be an exact indecomposable module category
over $\Rep(H)$.
\begin{itemize}
  \item[1.] There exists a subgroup $F\subseteq C_{p^2}$, a compatible datum
$(\xi,\alpha)$ and $W\subseteq V$ a subcomodule invariant under
the action of $F$ such that $\Mo\simeq {}_{\Ac(W,F, \xi,
\alpha)}\Mo$ as module categories.
  \item[2.] The left $H$-comodule algebras $\Ac(W,F, \xi, \alpha)$, $\Ac(W',F', \xi', \alpha')$ are equivariantly Morita equivalent if and only if $(W,F, \xi, \alpha)=(W',F', \xi', \alpha')$.
\end{itemize}\qed
\end{teo}

Given a compatible datum $(\xi, \alpha)$ with respect to $V$ and $C_p$  define the left $A(H, s)$-comodule algebra  $\widehat{\Ac}(V, \xi, \alpha)$ with underlying algebra equal to $\Ac(V,C_p, \xi, \alpha)$ and coaction $\widehat{\lambda}:\widehat{\Ac}(V, \xi, \alpha)\to A(H, s)\otk\widehat{\Ac}(V, \xi, \alpha)$ given by
$\widehat{\lambda}(a)=J_s\lambda(a) J^{-1}_s$ for all $a\in \widehat{\Ac}(V, \xi, \alpha)$. If $W\subseteq V$ is a $\ku C_{p^2}$-subcomodule invariant under the
action of $C_p$ define  $\widehat{\Ac}(W, \xi, \alpha)$ as the subalgebra of $\widehat{\Ac}(V, \xi, \alpha)$ generated by $W$ and $C_p$.

\medbreak

As a  consequence of Theorem \ref{main6} we have the following result.

\begin{teo} Let $\Mo$ be an exact indecomposable module category
over $\Rep(A(H, s))$.
\begin{itemize}
  \item[1.] There exists a compatible datum
$(\xi,\alpha)$ and $W\subseteq V$ a subcomodule invariant under
the action of $C_p$ such that there is an equivalence $\Mo\simeq {}_{\widehat{\Ac}(W, \xi, \alpha)}\Mo$ as $\Rep(A(H, s))$-module categories.
 \item[2.] The comodule algebras $\widehat{\Ac}(W, \xi, \alpha)$, $\widehat{\Ac}(W', \xi', \alpha')$ are equivariantly Morita equivalent if and only if $(W, \xi, \alpha)=(W', \xi', \alpha')$.
\end{itemize}\qed
\end{teo}

\end{document}